\documentclass[a4paper,11pt]{article}
\usepackage[utf8x]{inputenc}
\usepackage[T1]{fontenc}
\usepackage{latexsym,amsfonts,amssymb,mathrsfs,stmaryrd}
\usepackage{amsmath,amsthm,amsrefs}
\usepackage{thmtools}
\usepackage[all]{xy}
\usepackage{color}
\usepackage{enumerate}
\usepackage[pdftex]{graphicx}
\usepackage[pdftex,
colorlinks=true,
linkcolor=blue,
citecolor=green,
hyperindex,
plainpages=false,
bookmarks=true,
bookmarksopen=true,
bookmarksnumbered,]{hyperref}

\theoremstyle{plain}
\newtheorem{theorem}{Theorem}[section]
\newtheorem{corollary}[theorem]{Corollary}
\newtheorem{lemma}[theorem]{Lemma}
\newtheorem{proposition}[theorem]{Proposition}
\newtheorem{remark}[theorem]{Remark}

\theoremstyle{definition}
\newtheorem{definition}[theorem]{Definition}
\theoremstyle{remark}

\newtheorem*{theorem*}{Theorem \ref{mainthm}}

\numberwithin{equation}{section}

\renewcommand{\H}{\mathcal{H}}

\newcommand{\midd}{\;\middle|\;}
\newcommand{\diam}{{\rm diam}}
\newcommand{\dist}{{\rm dist}}
\newcommand{\Lip}{{\rm Lip}}
\newcommand{\id}{{\rm id}}

\begin{document}
\title{H\"older regularity at the boundary of two-dimensional sliding almost minimal sets}
\author{Yangqin FANG}
\newcommand{\Addresses}{{
\bigskip
\footnotesize

Yangqin FANG, \textsc{Laboratoire de Math\'ematiques d'Orsay, 
Universit\'e Paris-Sud (XI) 91405, Orsay Cedex, France}\par\nopagebreak
\textit{E-mail address}: \texttt{yangqin.fang@math.u-psud.fr}

%
}}

\date{}
\maketitle
\begin{abstract}
	In \cite{Taylor:1976}, Jean Taylor has proved a regularity theorem 
	away from boundary for Almgren almost minimal sets of dimensional two in 
	$\mathbb{R}^{3}$. It is quite important for understanding the soap films and
	the solutions of Plateau's problem away from boundary. In this paper, we 
	will give a regularity result on the boundary for two dimensional sliding 
	almost minimal sets in $\mathbb{R}^{3}$. It will be of use for understanding
	their boundary behavior.
\end{abstract}

\section{Introduction}
In \cite{David:2013,David:2012}, Guy David proposed to consider the Plateau
Problem with sliding boundary conditions. That is, given a closed set
$B\subset\mathbb{R}^{n}$, and an initial closed set
$E_0\supset B$, we aim to find a competitor $E$ such that 
$\H^d(E\setminus B)$ attains the infimum among all the competitors of $E_0$, 
where $d$ is an integer between $0$ and $n$. The sliding conditions 
seem very natural to Plateau's problem (or soap films). One of the advantages 
is that it may be easier to prove some regularity at the boundary. 
In fact, paper \cite{David:2014} paves the way to show the regularity. 

In the recent papers \cite{DGM:2014,DDG:2015}, the authors have proposed a direct
approach to the Plateau problem with sliding boundary conditions. Eventually
they proved an existence result. Which is that when $B$ is a closed set with
$\H^d(B)=0$, then there exists (at least) a sliding minimizer for the Plateau
problem with sliding boundary conditions.

The aim of the present paper is to study the regularity of the sliding minimizer.
In \cite{Taylor:1976}, Jean Taylor proved that any $2$-dimensional reduced 
almost minimal set in $\mathbb{R}^{3}$ is locally $C^1$-diffeomorphic to a 
minimal cone. Here we hope to prove a similar regularity result at boundary for 
sliding almost minimal set. That is, at the boundary it is locally 
$C^1$-diffeomophic to a sliding minimal cone. But unfortunately we do not 
prove the $C^1$-diffeomophic equivalence at this time. In this paper, we show
that if $E\subset\mathbb{R}^{3}$ is a reduced $2$-dimensional set with a 
sliding boundary condition given by a smooth $2$-dimensional surface, then 
$E$ is locally biH\"older equivalent to a sliding minimal cone.    

In \cite{David:2009} and \cite{David:2008}, Guy David has given a
new, more detailed, proof of a good part of Jean Taylor's regularity theorem 
for Almgren almost minimal sets of dimensional 2 in $\mathbb{R}^{3}$, and
generalized it to $\mathbb{R}^{n}$. At the same time, he proved a theorem
of almost monotonicity of density for almost minimal sets away from the
boundary. In fact, his proof of H\"older regularity relies on
the theorem of almost monotonicity of density and a Reifenberg parameterization.
In \cite{David:2014}, he established a very similar result as in
\cite{David:2009}, a theorem of almost monotonicity of density at the boundary
for sliding almost minimal sets. This allow us to prove the H\"older regularity 
of these sets on the boundary in some case. At the end of the paper, we will 
discuss how to use regularity result blow to prove some existence results.

Let us begin with some notation and definitions. A gauge function is a 
nondecreasing function $h:[0,\infty]\to[0,\infty]$ with $\lim_{t\to 0} h(t)=0$.
Let $\delta>0$ and an open set $U\subset\mathbb{R}^{n}$ be given.
A $\delta$-deformation in $U$ is a family of maps 
$\{ \varphi_t \}_{0\leq t\leq 1}$ from $U$ into itself such that 
\[
\varphi_1\text{ is Lipschitz and }\varphi_0=\id_{U},
\]
the function 
\[
[0,1]\times U\to U, (t,x)\mapsto \varphi_{t}(x)
\]
is continuous, $\widehat{W}$ is relatively compact in $U$ and
$\diam(\widehat{W})<\delta$, where
\begin{equation}\label{eq:0}
	\widehat{W}=\bigcup_{t\in [0,1]}\left( W_{t}\cup \varphi_t(W_t) \right), \ 
	W_t=\{ x\in U;\varphi_t(x)\neq x \}.
\end{equation}

We say that a relatively closed $d$-dimensional set $E\subset U$ is
$(U,h)$-almost-minimal if it satisfies
\[
\H^d(E\cap W_1)\leq \H^d(\varphi(E\cap W_1))+h(\delta)\delta^d,
\]
for any $\delta$-deformation $\{ \varphi_t \}_{0\leq t\leq 1}$. In
\cite{Taylor:1976}, Jean Taylor proved that if $U$ is an open set in
$\mathbb{R}^{3}$, $E$ is a reduced $(U,h)$-almost-minimal set and $h(r)\leq
cr^{\alpha}$, then for any $x\in U$, there is a small neighborhood of $x$ 
contained in $U$ and in this neighborhood, $E$ is $C^1$ diffeomorphic to a
$2$-dimensional minimal cone, while $2$-dimensional minimal cones are planes,
cones of type $\mathbb{Y}$ and cones of type $\mathbb{T}$. 

In this paper, we concentrate on boundary regularity, and always consider the
following sliding boundary conditions. Let $\Omega\subset\mathbb{R}^{n}$ be a
closed domain in $\mathbb{R}^{n}$. Let $L_1$ be a closed sets (it will be
consider as the sliding boundary).
\begin{definition}
	Let $U$ be an open set. For $\delta>0$, we say that a one parameter family 
	$\{ \varphi_t \}_{0\leq t\leq 1}$ of maps from $U$ into itself is a
	$\delta$-sliding-deformation in $U$, if it satisfies the following properties:
	$\varphi_0=\id_{U}$, $\varphi_1$ is Lipschitz,
	$(t,x)\mapsto \varphi_t(x)$ is continuous on $[0,1]\times U$,
	$\varphi_t(x)\in L_1$ for any $x\in L_1$ and any $t\in [0,1]$, 
	$\widehat{W}$ is relatively compact in $U$ and $\diam(\widehat{W})<\delta$.
\end{definition}

Let $E\subset \Omega$ be closed in $\Omega$; we say that a closed subset
$F\subset\Omega$ is a competitor of $E$ in $U$, if $F=\varphi_1(E)$ 
for some sliding deformation $\{ \varphi_t \}_{0\leq t \leq 1}$ in $U$.

\begin{definition}
	Let $E\subset\Omega$ be closed in $U$. We say that $E$ is
	$(U,h)$-sliding-almost-minimal, 
	if for each $\delta>0$ and each $\delta$-sliding-deformation 
	$\{\varphi_t\}_{0\leq t\leq 1}$, we have 
	\begin{equation}\label{eq:almostminimal}
		\H^d(E\cap W_{1})\leq\H^d(\varphi_{1}(E\cap
		W_{1}))+h(\delta)\delta^{d},
	\end{equation}
	where $W_1=\{ x\in U;\varphi_{1}(x)\neq x \}$.

	We say that $E$ is an $A_{+}$-sliding-almost-minimal set in $U$ if under the
	same circumstances,
	\[
	\H^d(E\cap W_1)\leq (1+h(\delta))\H^d(\varphi_1(E\cap W_1)).
	\]
\end{definition}

In the definition of $(U,h)$-sliding-almost-minimal set, we can replace
inequality \eqref{eq:almostminimal} by the inequality 
\[
\H^d(E\setminus\varphi_1(E))\leq \H^d(\varphi_1(E)\setminus
E)+h(\delta)\delta^d,
\]
at least if that $L_1$ is not too bed, see \cite{David:2014}.

When $\Omega$, $L_1$ and gauge function $h$ are clear, and $U=\mathbb{R}^{n}$.
For simplicity, we may say that an $(U,h)$-sliding-almost-minimal set $E$ is 
sliding almost minimal (in $\Omega$ with sliding boundary $L_1$). It quite easy
to see that for any $(U,h)$-sliding-almost-minimal set $E$,
$E\setminus L_1$ is $(U\setminus L_1,h)$-almost-minimal. 

We say that $E$ is (sliding) minimal in $U$ if
it is (sliding) almost minimal with gauge function $h=0$, that is, 
\[
\H^d(E\cap W_1)\leq \H^d(\varphi_1(E\cap W_1))
\]
or
\[
\H^d(E\setminus \varphi_1(E))\leq \H^d(\varphi_1(E)\setminus E)
\]
for any (sliding) deformation $\{ \varphi_t \}_{0\leq t\leq 1}$ in $U$.

We say that a $d$-dimensional set $E$ is reduced if $E=E^{\ast}$, where
\[
E^{\ast}=\{ x\in E\mid \H^d(E\cap B(x,r))>0 \text{ for every }r>0\}.
\]
We can prove that 
\[
\H^d(E\setminus E^{\ast})=0
\]
and that $E^{\ast}$ is also (sliding) almost minimal when $E$ is (sliding)
almost minimal, see for instance \cite{David:2009, David:2014}. In this paper,
we always assume that a sliding almost minimal set is reduced.

For any set $E$, any point $x\in E$ and any radius $r>0$,
we set 
\[
\theta_E(x,r)=\frac{\H^d(E\cap B(x,r))}{\omega_d r^d},
\]
where $\omega_d$ denote the Hausdorff measure of $d$-dimensional unity ball.
If the limit 
\[
\lim_{r\to 0}\theta_E(x,r)
\]
exists, we will denote it by $\theta_E(x)$, and call it the density
of $E$ at the point $x$. When $E$ is given, and there is no danger of confusion,
we may drop the subscript $E$ and denote it by $\theta(x)$. A property of almost
monotonicity of density for sliding almost minimal set will be often used.
That is, Proposition 5.27 in \cite{David:2009} and Theorem 28.7 in
\cite{David:2014}. We now put them together, it can be stated rough as follows:
\begin{equation}\label{eq:AMDS}
	\theta_E(x,r)e^{\lambda A(r)} \text{ is a nondecreasing function of }r,
\end{equation}
when $r$ small, where $E$ is a sliding almost minimal set with gauge function
$h$, and $A(r)=\int_{0}^{r}h(2t)\frac{dt}{t}$. Let's refer to
\cite{David:2014} and \cite{David:2009} for more detailed statement.
We get that from \eqref{eq:AMDS} that when $E$ is a sliding almost minimal set,
for any $x\in E$, the density $\theta_E(x)$ exists. 

A blow-up limit of a set $E$ at $x\in E$ is any closed set in $\mathbb{R}^{n}$
that can be obtained as the limit of a sequence $\{r_k^{-1}(E-x)\}$ with
$\lim_{k\to\infty}r_k=0$.

A set $E$ in $\mathbb{R}^{n}$ is called a cone centered at origin $0$ if for 
any $x\in E$ and any $t\geq 0$, $tx\in E$. In general, a cone is a translation 
of a cone centered at origin.

Suppose that $E$ is a sliding almost minimal set, and $x\in E$. If $x$ is not
contained in the sliding boundary, then any blow-up limit of $E$ at $x$ is a
minimal cone in $\mathbb{R}^{n}$, see \cite{David:2009}; if $x$ is in the
sliding boundary, then any blow-up limit of $E$ at $x$ is a sliding minimal
cone, see \cite{David:2014}. We also refer to \cite{David:2009} and
\cite{David:2014} for the basic properties of blow-up limits.

The main theorem of the paper is following:
\begin{theorem*}
	Let $\Sigma\subset\mathbb{R}^{3}$ be a connected closed set such
	that the boundary $\partial\Sigma$ is a two-dimensional $C^1$ submanifold. 
	Suppose that $x$ is a point in $\partial\Sigma$, $U$ is a neighborhood of
	$x$, $E\subset \Sigma$ is an $(U,h)$-sliding-almost-minimal set with sliding
	boundary $\partial\Sigma$ and $E\supset \partial\Sigma$. Then for each small
	$\tau>0$, we can find a radius $\rho>0$, a sliding minimal cone $Z$ in 
	$\Omega$ with sliding boundary $L_1$ and a biH\"older map 
	$\phi: B(x,3\rho/2)\cap \Omega\to B(x,2\rho)\cap \Sigma$ such that
	\[
		\begin{gathered}
			\phi(z)\in \partial\Sigma \text{ for } z\in L_1,\ 
			\| \phi-\id \|_{\infty}\leq 3\tau,\\
			C\left\vert z-y \right\vert^{1+\tau}\leq \left\vert
			\phi(z)-\phi(y) \right\vert\leq C^{-1}\left\vert z-y
			\right\vert^{\frac{1}{1+\tau}} ,\\
			B(x,\rho)\cap \Sigma\subset \phi\left( B\left( x,\frac{3\rho}{2} \right)\cap
			\Omega \right)\subset B(x,2\rho)\cap \Sigma,\\
			E\cap B(x,\rho)\subset \phi\left(Z\cap B\left( x,\frac{3\rho}{2}
			\right)\right)\subset E\cap B(x,2\rho).
		\end{gathered}
	\]
	where $\Omega$ is a half space and $L_1$ is the boundary of $\Omega$.
\end{theorem*}
The list of $2$-dimensional sliding minimal cones which contain $L_1$ is not
complicated, which is following: $L_1$, cones $L_1\cup Z$, where $Z$ is a cone
of type $\mathbb{P}_{+}$ or $\mathbb{Y}_{+}$. See Section \ref{se:2dimcones}
for precise definitions and see Theorem \ref{thm:MCCB} for a precise
statement.

It seems to be a reasonable
condition for soap film that $E\supset \partial\Sigma$. In soap film
experiments, if we dip a shape of wire into some soapy water, when we pull it
out we shall obtain a surface created by the soap film. The wire is considered
as the sliding boundary, and the surface is consider as a sliding almost
minimal set. Actually, this surface seems to contain the wire. Thus the 
assumption $E\supset\partial\Sigma$ seems natural to the author. 

It would be also very interesting to consider the regularity at the boundary
of sliding almost minimal sets which do not necessarily contain the boundary. 
But unfortunately, without the assumption $E\supset
\partial\Sigma$, we do not have a satisfactory result. Because in this case, the
blow-up limits of $E$ at a point $x\in E\cap\partial\Sigma$ could be cones of
type $\mathbb{T}_{+}$ or cones of type $\mathbb{V}$. See Section
\ref{se:2dimcones} for a precise definitions of these cones. When a blow-up 
limit is a cone of type $\mathbb{V}$, we will meet trouble (Figure \ref{fig:buv} is
an example of potential soap film for which regularity seems difficult to
prove).
\begin{figure}[!htb]
	\centering
	\includegraphics[width=5in]{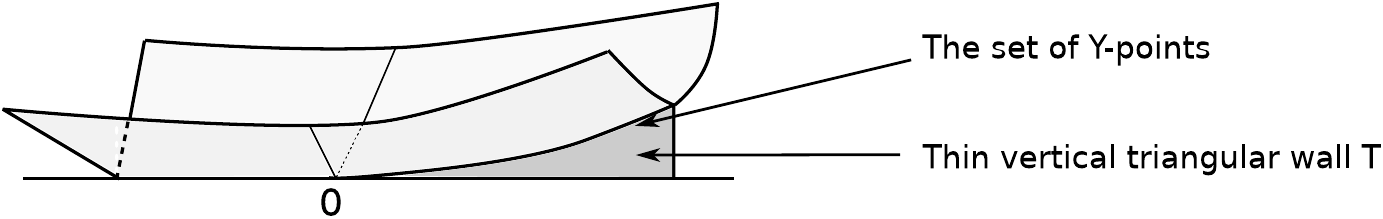}
	\caption{blow-up limit at $0$ is a cone of type $\mathbb{V}$}
	\label{fig:buv}
\end{figure}

\section{One dimensional sliding minimal sets in a half
plane}\label{se:1dimcones}
In this section we discuss one dimensional sliding minimal sets in a half plane. 
We discuss the one dimensional case, because it is very easy, and the list of 
one dimensional sliding cones will be used to classify the two dimensional 
sliding minimal cones in a half space. For simplicity, we assume that
\begin{equation}\label{eq:domainboundary}
	\begin{gathered}
		\Omega=\{ (x,y)\in \mathbb{R}^{2}\mid y\geq 0 \},\\
		L_1=\{ (x,0)\in\mathbb{R}^{2}\mid x\in \mathbb{R} \}.\\
	\end{gathered}
\end{equation}

For any $t\in\mathbb{R}$, and any $\alpha\in (0,\frac{\pi}{2})$, we set 
\[
P_t=\{ (t,y)\mid y\geq 0 \}\] 
and 
\[
V_{\alpha,t}=\{ (x,y): y=\left\vert x\tan\alpha+t \right\vert\}.
\] 
It is very easy to see that the set $P_t$ and $P_t\cup L_1$ are sliding
minimal. It is also not hard to see that $V_{\alpha,t}$ is minimal if and
only if $0<\alpha\leq\frac{\pi}{6}$.

\begin{lemma}\label{listofonedimensionalminimalcone}
	Let $\Omega$ and $L_1$ be as in \eqref{eq:domainboundary}. Suppose that 
	$E$ is a minimal cone in $\Omega$ with sliding boundary $L_1$ centered at 
	$0$. Then $E$ is one of $L_1$, $P_0$, $P_0\cup L_1$ and $V_{\alpha,0}$ for
	some $\alpha\in (0,\frac{\pi}{6})$.
\end{lemma}
\begin{proof}
	Let $K=E\cap \partial B(0,1)$. We note that $K$ is a finite set because otherwise 
	$\H^1(E\cap B(0,1))=\infty$. Write $K=\{ a_1,\cdots, a_n \}$, and denote by
	$l_i$ the ray form $0$ through the point $a_i$. Suppose
	$a_i,a_j\in \Omega\setminus L_1$, $i\neq j$. Similarly to (10.3) in Lemma 10.2
	in \cite{David:2009}, we can get that 
	\[
	{ \rm Angle }(l_i,l_j)\geq \frac{2\pi}{3}.
	\]
	Therefore, there are at most four point in $K$.

	Case 1, if there is only one point in $K$, i.e. $K=\{ a_1 \}$. It is easy to
	see that $a_1\neq (1,0)$ and $a_1\neq (-1,0)$. If $a_1=(0,1)$, it is very easy
	to see that $E$ is minimal. If $a_1\neq (0,1)$, we put $a_1=(x,y)$, then 
	\[
	E'=\{ (x,ty)\mid 0\leq t \leq 1 \}\cup \{ (tx,ty)\mid t\geq 1 \}
	\]
	is a competitor of $E$, and $\H^1(E')<\H^1(E)$. Then $E$ could not be minimal.
	In this case, $E=P_0$ is a ray which is perpendicular to $L_1$.

	Case 2, there are two points in $K$, i.e. $K=\{ a_1,a_2 \}$. If $a_1=(-1,0)$
	and $a_2=(1,0)$, then $E=L_1$. If $a_1=(-1,0)$ and $a_2\neq (1,0)$, then
	\[
	E''=\left(E\setminus B(0,1)\right)\cup [a_1,a_2]
	\]
	is a competitor of $E$ and $\H^1(E'')<\H^1(E)$. Then $E$ could not be minimal.
	If $a_2=(1,0)$ and $a_1\neq (-1,0)$, for the same reason as before, $E$ is
	not minimal. If $a_1,a_2\not\in \{ (-1,0),(1,0) \}$, we put 
	$a_1=(\cos\alpha_1,\sin\alpha_1)$, $a_2=(\cos\alpha_2,\sin\alpha_2)$ and 
	$\widetilde{a}_2=(\cos\alpha_2,-\sin\alpha_2)$ with $0<\alpha_2<\alpha_1<\pi$,
	then 
	\[
	\H^1([a_1,0]\cap [0,a_2])\geq \H^1([a_1,\widetilde{a}_2]),	
	\]
	and with equality if and only if $\alpha_1+\alpha_2=\pi$. It means that when
	$a_2\neq (-\cos\alpha_1,\sin\alpha_1)$, $E$ could not be minimal. We now
	suppose that $a_2= (-\cos\alpha_1,\sin\alpha_1)$. Then $E=V_{\alpha_1,0}$,
	and $0<\alpha_1\leq \frac{\pi}{6}$ because $E$ is minimal.

	Case 3, there are three point in $K$. Write $K=\{ a_1,a_2,a_3 \}$,
	$a_1=(x_1,y_1)$, $a_2=(x_2,y_2)$ and $a_3=(x_3,y_3)$, $x_1<x_2<x_3$. If
	$x_1\neq -1$ and $x_3\neq 1$, then 
	\begin{equation}\label{eq:05}
		\mbox{Angle}(l_1,l_2)\geq \frac{2\pi}{3}
		\mbox{ and } \mbox{Angle}(l_2,l_3)\geq \frac{2\pi}{3};
	\end{equation}
	that is impossible. If $x_1=-1$ and $x_3\neq 1$, then 
	\[
	\mbox{Angle}(l_2,l_3)\geq \frac{2\pi}{3},
	\]
	thus $-1<x_2<-\frac{\sqrt{3}}{2}$. We can see that 
	\[
	E'''=\left( E\setminus B(0,1) \right)\cup [0,a_1]\cup [(x_2,0),a_2]\cup [0,a_3]
	\]
	is a competitor of $E$, and $\H^1(E''')<\H^1(E)$, thus $E$
	could not be minimal. Similarly, we can see that we cannot have $x_1\neq 1$
	and $x_3=-1$. We now suppose that $x_1=-1$ and $x_3=1$, i.e. $a_1=(-1,0)$ and
	$a_3=(1,0)$. If $x_2\neq 0$, then 
	\[
	E''''=L_1\cup \{ (x_2,ty_2)\mid 0\leq t\leq 1 \}\cup \{ (tx_2,ty_2)\mid t\geq 1 \}
	\]
	is a competitor of $E$, and $\H^1(E'''')<\H^1(E)$, $E$ is not
	minimal, impossible! If $x_2=0$, then $a_2=(0,1)$. That is, $E=P_0\cup L_1$.

	Case 4, there are four point in $K$. If there are at least three point in
	$\partial B(0,1)\cap \Omega\setminus L_1$, similarly to \eqref{eq:05}, that is
	impossible. Thus there at most two point in 
	$\partial B(0,1)\cap \Omega\setminus L_1$, so there are exactly two point in
	$\partial B(0,1)\cap \Omega\setminus L_1$ and exactly two point in $\partial
	B(0,1)\cap \Omega\cap L_1$. We put $K=\{ a_1,a_2,a_3,a_4 \}$, $a_1=(-1,0)$, 
	$a_2=(x_2,y_2)$, $a_3=(x_3,y_3)$, $a_4=(1,0)$, $\widetilde{a}_{2}=(x_2,0)$ and
	$\widetilde{a}_{3}=(x_3,0)$. Then  
	\[
	\H^1([a_2,0]\cap [0,a_3])> \H^1([a_2,\widetilde{a}_2])+\H^1([a_3,\widetilde{a}_3]),
	\]
	which means that $E$ could not be minimal.
\end{proof}

\begin{proposition}\label{listofonedimensionalminimalset}
	Let $\Omega$ and $L_1$ be as in \eqref{eq:domainboundary}. Suppose that $E$ is a 
	sliding minimal set in $\Omega$ with sliding boundary $L_1$, and $E\supset L_1$. 
	Then either $E=L_1$ or $E=P_t\cup L_1$ for some $t\in \mathbb{R}$. 
\end{proposition}
\begin{proof}
	For $r>0$, we put $E_r=\frac{1}{r}E$, $A_r=\Omega\cap \overline{B(0,r)}$ 
	and $S_r=\Omega\cap \partial B(0,r)$.

	We claim that there exists a sequence $\{ r_n \}$, such that $r_n\to \infty$
	and there are at most three point in $E\cap S_{r_n}$. 

	Since $E$ is sliding minimal, we have that $\theta_{E}(0,r)$ is nondecreasing and
	bounded, see \cite[Theorem 28.4]{David:2014}. Thus for any $\varepsilon>0$, 
	we can find $r_{\varepsilon}>0$ such 
	that $\theta_{E}(0,r)\geq \theta_{E}(0,\infty)-\varepsilon$ for $r\geq
	r_\varepsilon$, where we denote $\theta_{E}(0,\infty)=\lim\limits_{r\to 
	\infty}\theta_{E}(0,r)$. We can easily see that
	$\theta_{E_r}(0,t)=\theta_{E}(0,rt)$. If we take $r>2r_{\varepsilon}$, then 
	\[
	\theta_{E_r}(0,t)=\theta_{E}(0,rt)\geq
	\theta_{E}(0,\infty)-\varepsilon=\theta_{E_r}(0,\infty)-\varepsilon,\
	\forall t\geq \frac{1}{2}.
	\]

	We now let $\tau$ with $0<\tau<\frac{1}{2}$ and $\varepsilon$ be as in 
	Proposition 30.3 in \cite{David:2014}. We take $t_0>2$ and apply Proposition 
	30.3 in \cite{David:2014}, and get that there is a minimal cone $T$ centered
	at $0$ such that 
	\begin{gather}
		\dist(y,T)\leq \tau t_0,\ \mbox{ for } y\in E_r\cap
		B(0,t_0-\tau)\setminus B\left(0,\frac{1}{2}+\tau\right),\notag\\
		\dist(z,E_r)\leq \tau t_0,\ \mbox{ for } z\in T\cap
		B(0,t_0-\tau)\setminus B\left(0,\frac{1}{2}+\tau\right),\notag\\
		\left\vert \H^1(E_r\cap B(y,u))-\H^1(T\cap B(y,u)) \right\vert\leq \tau
		t_0\label{subeq:one}\\ 
		\mbox{for any }B(y,u)\subset B(0,t_0-\tau)\setminus
		B\left(0,\frac{1}{2}+\tau\right),\mbox{ and }\notag\\
		\left\vert \H^1(E_r\cap B(0,t))-\H^1(T\cap B(0,t)) \right\vert\leq \tau t,\
		\forall \frac{1}{2}+\tau\leq t\leq t_0-\tau.\label{subeq:two}
	\end{gather}

	If we put $N(t)=\#\left( \partial B(0,t)\cap E \right)$, then by Lemma 8.10
	in \cite{David:2009} or Theorem 3.2.22 in \cite{Federer:1969},
	\[
	\frac{1}{s}\int_{0}^{s}N(t)dt\leq \frac{1}{s}\H^1(E\cap
	B(0,s)).
	\]
	Combining this with \eqref{subeq:two}, we can get
	that, for $\left( \frac{1}{2}+\tau \right)r\leq s\leq (t_0-\tau)r $, 
	\[
	\frac{1}{s}\int_{0}^{s}N(t)dt \leq \frac{1}{s}\H^1(T\cap B(0,s))+\tau\leq
	3+\tau.
	\]
	But $t_0$ can be chosen arbitrarily large, thus we can find a sequence
	$\{s_n\}_{n=1}^{\infty}$ such that $s_n\to \infty$ and $N(s_n)<4$. Since
	$L_1\subset E$, we have that $N(s)\geq 2$ for any $s>0$, thus $2\leq
	N(s_n)\leq 3$.

	If $N(s)=2$ for some $s>0$, by minimality of $E$, we can get that
	\[
	E\cap B(0,s)=L\cap B(0,s).
	\]

	If $N(s)=3$ for some $s>0$, we suppose that 
	\[E\cap \partial B(0,s)=\{ (-s,0),X_s,(s,0) \},\ X=(x_s,y_s).\]
	If the points $X_s$ and $0$ are not in the same component of $E\cap
	\overline{B(0,s)}$, then by minimality of
	$E$, we can see that $X_s$ is the only point in the component of $E\cap
	\overline{B(0,s)}$ which contains the point $X_s$, and 
	\[
	E\cap B(0,s)=L\cap B(0,s).
	\]
	If the points $X_s$ and $0$ are in the same component of $E\cap
	\overline{B(0,s)}$, then there is a path in $E\cap \overline{B(0,s)}$ from 
	$X_s$ to $0$, we denote it by $\gamma:[0,1]\to E\cap \overline{B(0,s)}$ with
	$\gamma(0)=X_s$ and $\gamma(1)=0$. Let 
	\[
	v_0=\inf\{ v\in [0,1]\mid \gamma(v)\in L_1 \}.
	\]
	Then 
	\[
	E'=[X_s,\gamma(v_0)]\cup \left( E\setminus B(0,s) \right)\cup \left( L_1\cap
	B(0,s) \right)
	\]
	and 
	\[
	E''=[(x_s,y_s),(x_s,0)]\cup \left( E\setminus B(0,s) \right)\cup \left( L_1\cap
	B(0,s) \right)
	\]
	are competitors of $E$, and 
	\[
	\H^1(E''\cap B(0,s))\leq \H^1(E'\cap B(0,s))\leq \H^1(E\cap B(0,s)).
	\]
	By minimality of $E$, we get that $E''=E'=E$.
	We put $\gamma(v_0)=(t_s,0)$, then
	\[
	E\cap B(0,s)=(P_{t_s}\cup L_1)\cap B(0,s).
	\]

	If there exists a sequence $\{ n_{k} \}_{k=1}^{\infty}$ such that
	$N(s_{n_k})=2$, then 
	\[
	E\cap B(0,s_{n_k})=L_1\cap B(0,s_{n_k})
	\]
	for any $k\geq 1$; thus we get that $E=L_1$.

	If there exist an integer $n_0\geq 1$ such that $N(s_n)=3$ for $n\geq n_0$,
	then
	\[
	E\cap B(0,s_n)=(P_{t_{s_n}}\cup L_1) \cap B(0,s_n),
	\]
	and $t_{s_n}=t_{s_{n_0}}$ for any $n\geq n_0$. By putting
	$t=t_{s_{n_0}}$, we get that $E=P_t\cup L_1$.
\end{proof}

\section{Two dimensional minimal cone with sliding
boundary}\label{se:2dimcones}

In this section we consider a simple case in $\mathbb{R}^{3}$: our domain 
$\Omega$ is a half space, and the boundary $L_1$ is the plane which is the 
boundary of $\Omega$. In the domain $\Omega$, we will see what does a sliding
minimal cone look like. For simplicity, we assume that 
\begin{equation}\label{simple}
	\begin{gathered}
		\Omega  =\{ x=(x_1,x_2,x_3)\in\mathbb{R}^{3}\mid x_3 \geq 0 \},\\
		L_{1} =\{ x=(x_1,x_2,x_3)\in\mathbb{R}^{3}\mid x_3 = 0 \}.
	\end{gathered} 
\end{equation}

Let us refer to paper \cite{David:2009} for the definition of cones of type
$\mathbb{Y}$ and $\mathbb{T}$.
We say that a cone $Z\subset\Omega$ is of type $\mathbb{P}_{+}$, if $Z$ 
is a closed half plane which is perpendicular to $L_1$ and through $0$, 
i.e. the intersection of $\Omega$ with a plane which is through $0$ and meets
$L_1$ perpendicularly; similarly we say that a cone $Z\subset\Omega$ is of type
$\mathbb{Y}_{+}$ if it is a intersection of $\Omega$ with a cone in
$\mathbb{R}^{3}$ of type $\mathbb{Y}$which is perpendicular to $L_1$. Recall
that a cone $Z$ in $\mathbb{R}^{3}$ of type $\mathbb{Y}$ is the union of there
half planes bounded by a line $\ell$, called the spine of $Z$. Here we say 
that a cone of type $\mathbb{Y}$ is perpendicular to $L_1$, if the spine of 
the cone is perpendicular to $L_1$. We will check that cones of type
$\mathbb{P}_{+}$ or $\mathbb{Y}_{+}$ are sliding minimal.

Let $Z$ be a cone of type $\mathbb{T}$, we say that $Z$ is perpendicular to
$L_1$ if the center of $Z$ locates at the origin and 
$Z\cap \overline{\Omega^c}$ is a cone of type $\mathbb{Y}_{+}$ in the 
domain $\overline{\Omega^c}$.

We say that a cone $Z\subset\Omega$ is of type $\mathbb{T}_{+}$ if it
is the intersection of $\Omega$ with a cone of type $\mathbb{T}$ which is
perpendicular to $L_1$. In this paper, we do not discuss whether or not a cone
of type $\mathbb{T}_{+}$ is sliding minimal.

A cone $Z\subset \Omega$ is called of type $\mathbb{V}$ is it can be written
as $Z=\mathscr{R}(\mathbb{R}\times V_{\alpha,0})$ where $\mathscr{R}$ is a 
rotation which maps $L_1$ into $L_1$, $V_{\alpha,0}$ is cone in a half plane 
defined as in Section \ref{se:1dimcones} and $0<\alpha<\frac{\pi}{2}$. 

\begin{lemma} \label{le:PLMC}
	Let $\Omega$, $L_{1}$ be as in \eqref{simple}. If $Z$ is a cone of type
	$\mathbb{P}_{+}$ or $\mathbb{Y}_{+}$, then $Z$ is a sliding minimal cone.
	If $Z=Z'\cup L_1$ and $Z'$ is a sliding minimal cone of type $\mathbb{P}_{+}$ 
	or $\mathbb{Y}_{+}$, then $Z$ is also a sliding minimal cone.
\end{lemma}
\begin{proof}
	Suppose that $Z$ is of type $\mathbb{P}_{+}$ or
	$\mathbb{Y}_{+}$, which is not sliding minimal. Then there is a 
	competitor of $Z$, say $E$, such that 
	\[
	\H^2(E\setminus Z)<\H^2(Z\setminus E).
	\]
	Let $\sigma:\mathbb{R}^{3}\to\mathbb{R}^{3}$ be the reflection with respect
	to the plane $L_1$. That is, for any $(x_1,x_2,x_3)\in \mathbb{R}^{3}$,
	$\sigma(x_1,x_2,x_3)=(x_1,x_2,-x_3)$.
	Then $\widetilde{E}=E\cup \sigma(E)$ is a competitor of $\widetilde{Z}=Z\cup
	\sigma(Z)$, and 
	\[
	\H^2(\widetilde{E}\setminus \widetilde{Z})<\H^2(\widetilde{Z}\setminus
	\widetilde{E}).
	\]
	But we know that $\widetilde{Z}$ is a plane or a cone of type $\mathbb{Y}$,
	which is minimal in $\mathbb{R}^{3}$, that gives a contradiction.

	Now suppose that $Z=Z'\cup L_1$, where $Z'$ is cone of type $\mathbb{P}_{+}$ or
	$\mathbb{Y}_{+}$. Let $E$ be any competitor of $Z$. Suppose that $E$
	coincide with $Z$ out of the ball $B(0,r/2)$.
	Let $\pi:\mathbb{R}^{3}\to\mathbb{R}$ be the function defined by
	$\pi(x_1,x_2,x_3)=x_3$. By using Lemma 8.10 in \cite{David:2009} or Theorem
	3.2.22 in \cite{Federer:1969}, we get that 
	\[
	\int_{E\cap B(0,r)}apJ_{m}\pi(z)d\H^2(z)=\int_{y\in\mathbb{R}}\int_{z\in
	\pi^{-1}(y)}1_{B(0,r)\cap E}(z)d\H^1(z)d\H^1(y),
	\]
	where $apJ_m\pi(x)$ is the approximate Jacobian, see \cite{Federer:1969}.
	We can check that $apJ_{m}\pi(z)\leq 1$ for any $z\in E$. Thus 
	\[
	\H^2(E\cap B(0,r)\setminus L_1)\geq \int_{0}^{r}\int_{z\in
	\pi^{-1}(y)}1_{E\cap B(0,r)}(z)d\H^1(z)d\H^1(y).
	\]
	For any $0<y<r$, $\pi^{-1}(y)$ is a plane, and $Z\cap \pi^{-1}(y)$ is a line
	or a $Y$ in this plane, so it is minimal in the plane. But $E\cap
	\pi^{-1}(y)$ coincide with $Z\cap \pi^{-1}(y)$ out of the ball $B(0,1)$, and 
	it is not hard to check that $E\cap \pi^{-1}(y)$ is connected. Thus
	\begin{align*}
		\int_{z\in \pi^{-1}(y)}1_{E\cap B(0,r)}(z)d\H^1(z)&	=\H^1(E\cap B(0,1)\cap
		\pi^{-1}(y))\\
		&\geq \H^1(Z\cap B(0,1)\cap \pi^{-1}(y))\\
		&=\int_{z\in \pi^{-1}(y)}1_{Z\cap B(0,r)}(z)d\H^1(z),
	\end{align*}
	hence 
	\[
	\H^2(E\cap B(0,r)\setminus L_1)\geq \int_{0}^{r}\int_{z\in
	\pi^{-1}(y)}1_{Z\cap B(0,r)}(z)d\H^1(z)d\H^1(y).
	\]
	Since $Z=Z'\cap L_1$, and $Z'$ is a cone of type $\mathbb{P}_{+}$ or
	$\mathbb{Y}_{+}$, we have that 
	\[
	\H^1(Z\cap B(0,r)\setminus L_1)=\int_{0}^{r}\int_{z\in
	\pi^{-1}(y)}1_{Z\cap B(0,r)}(z)d\H^1(z)d\H^1(y).
	\]
	We get that 
	\[
	\H^2(E\cap B(0,r)\setminus L_1)\geq \H^1(Z\cap B(0,r)\setminus L_1),
	\]
	thus
	\[
	\H^2(E\setminus Z)\leq \H^2(Z\setminus E),
	\]
	and $Z$ is minimal.
\end{proof}

Let $Q$ be any convex polyheddron, $x$ be a point in the interior of $Q$. If
$F\subset Q$ is a compact set with $x\not\in F$, then we can find a Lipschitz
map
\begin{equation}\label{eq:radialprojection}
	\Pi_{Q,x}:\mathbb{R}^{3}\to\mathbb{R}^{3}
\end{equation}
such that 
\begin{equation}\label{eq:radialprojection2}
	\Pi_{Q,x}\vert_{Q^c}=\id_{Q^c},\ \Pi_{Q,x}(E)\subset\partial Q.
\end{equation}
Indeed, we take a very small ball $B(x,r)$ such that $B(x,r)\cap F=\emptyset$,
and consider the map $\varphi:\mathbb{R}^{3}\setminus
B(x,r)\to\mathbb{R}^{3}$ defined by 
\[
	\varphi(y)=\begin{cases}
		y,&y\in Q^c;\\
		\{ ty+(1-t)x\mid t\geq 0 \}\cap \partial Q,&x\in Q\setminus B(x,r).
	\end{cases}
\]
$\varphi$ is Lipschitz on $\mathbb{R}^{3}\setminus B(x,r)$. By the
Kirszbraun's theorem \cite[2.10.43]{Federer:1969}, we can find a Lipschitz
map $\Pi_{Q,x}:\mathbb{R}^{3}\to\mathbb{R}^{3}$ such that 
\[
	\Pi_{Q,x}\vert_{\mathbb{R}^{3}\setminus B(x,r)}=\varphi.
\]

\begin{lemma}
	Let $\Omega$, $L_{1}$ be as in \eqref{simple}. If $Z'$ is a cone of type
	$\mathbb{T}_{+}$, then the cone $Z=L_1\cup Z'$ is not minimal.
\end{lemma}
Here we do not want to talk about whether or not a cone of type
$\mathbb{T}_{+}$ is minimal, it is not so obvious. Recall that a cone of
type $\mathbb{T}$ has six faces, that meet by sets of three and with
$120^{\circ}$ angles along four edges (half lines emanating from the center).
\begin{proof}
	We put $O=(0,0,0)$,
	$A_1=(\frac{2\sqrt{2}}{3},0,0)$,
	$A_2=(-\frac{\sqrt{2}}{3},\frac{\sqrt{6}}{3},0)$, 
	$A_3=(-\frac{\sqrt{2}}{3},-\frac{\sqrt{6}}{3},0)$,
	$B_1=(\frac{2\sqrt{2}}{3},0,\frac{1}{3})$,
	$B_2=(-\frac{\sqrt{2}}{3},\frac{\sqrt{6}}{3},\frac{1}{3})$, 
	$B_3=(-\frac{\sqrt{2}}{3},-\frac{\sqrt{6}}{3},\frac{1}{3})$. 
	We denote by $C$ the triangular prism	$A_1A_2A_3B_1B_2B_3$, by $\Gamma$ the
	union of eight edges of $C$. Without loss of generality, we assume that 
	\[
	Z=\bigcup_{t\geq 0}t\Gamma.
	\]
	We denote by $F_0$, $F_1$, $F_2$ and $F_3$ the faces $A_1A_2A_3$,
	$A_3A_1B_1B_3$, $A_1A_2B_2B_1$ and $A_2A_3B_3B_2$ of the prism $C$ respectively.
	Consider $\tilde{Z}=(Z\setminus C) \cup F_0\cup F_1\cup F_2\cup F_3$. We
	will show that $\tilde{Z}$ is a competitor of $Z$. 

	We take $x_0=(0,0,\frac{1}{4})$, then $x_0$ is in the interior of the
	triangular prism. We take a Lipschitz map $\Pi_{C,x_0}$ as in
	\eqref{eq:radialprojection}.	Then $\tilde{Z}=\Pi_{C,x_0}(Z)$ is a
	competitor of $Z$.

	Since $Z\cap C$ consists of faces (triangles) $A_1A_2A_3$, $OA_1B_1$,
	$OA_2B_2$, $OA_3B_3$, $OB_1B_2$, $OB_2B_3$ and $OB_3B_1$.	By a simple	
	calculation, we can get that 
	\[
	\H^2(Z\cap C)=\frac{4\sqrt{2}+\sqrt{3}}{3}.
	\]
	Similarly, $\title{Z}\cap C$ consists of faces $F_0$, $F_1$, $F_2$ and $F_3$,
	thus
	\[
	\H^2(\tilde{Z}\cap C)=\frac{2\sqrt{6}+\sqrt{3}}{3}.
	\]
	Therefore 
	\[
	\H^2(\tilde{Z}\setminus Z)<\H^2(Z\setminus\tilde{Z}),
	\]
	and $Z$ is not minimal.
\end{proof}

\begin{definition}[{\cite[Definition 2.1]{Morgan:1994}}]
	Let $B_0$ be a closed subset of $\mathbb{R}^{n}$. Let $\delta,c,\alpha>0$ be
	given. We say that a nonempty bounded subset $S\subset \mathbb{R}^{n}\setminus B_0$
	is $d$-dimensional $({\bf M},cr^\alpha,\delta)$-minimal relative to $B_0$ if 
	\[
	\H^d(S)<\infty,\ S={\rm supp}(\H^d\llcorner S)\setminus B_0,
	\]
	and 
	\[
	\H^d(S\cap W)\leq (1+cr^\alpha)\H^d(\varphi(S\cap W))
	\]
	whenever $\varphi:\mathbb{R}^{n}\to \mathbb{R}^{n}$ is Lipschitz with
	$\diam(W\cup \varphi(W))=r<\delta$ and $\dist(W\cup \varphi(W),B_0)>0$, where
	$W=\{ z\in\mathbb{R}^{n}\mid \varphi(z)\neq z \}$.
\end{definition}

When $B_0=L_1$, $h(r)=cr^{\alpha}$, $E\subset \Omega$ is a reduced
bounded $(U,h)$-$A_{+}$-sliding-almost-minimal set, then it is very easy to
see that $E$ is also $({\bf M},cr^\alpha,\delta)$-minimal relative to $B_0$.
With help of this property, we can use a result of Morgan \cite[Regularity
Theorem 3.8]{Morgan:1994}, which is stated as follows.
\begin{theorem}\label{thm:regularityonsphere}
	Fix $\delta,c,\alpha>0$. Let $B_0$ be a closed subset of $\mathbb{R}^{n}$.
	Let $S$ be a one-dimensional $({\bf M},cr^\alpha,\delta)$-minimal set with
	respect to $B_0$. Then $S$ consists of $C^{1,\alpha/2}$ curves that can only
	meet in three at isolated points of $\mathbb{R}^{n}\setminus B_0$ and with
	$120^\circ$ angles .
\end{theorem}

Let $E$ be a sliding minimal cone in $\Omega$.  Set $K=\partial B(0,1)\cap
E$, $S=K\setminus L_1$. We want to show that $S$ is $({\bf
M},cr^\alpha,\delta)$-minimal with respect to $B_0=L_1$ for some
$\alpha,c,\delta>0$.

\begin{proposition}\label{prop:regularityonshpere}
	Let $\Omega,L_1$ be as in \eqref{simple}, $B_0=L_1$. Let $E$ be a reduced
	sliding minimal cone in $\Omega$, $K=\partial B(0,1)\cap E$. 
	If $K\setminus L_1\neq \emptyset$, then $K$ is 
	$A_{+}$-sliding-almost-minimal for some gauge function $h$ such that
	$h(r)=cr$ for $r<\frac{1}{100}$.
\end{proposition}
\begin{proof}
	Let $\{\varphi_t\}_{0\leq t\leq 1}$ be a deformation with 
	$\diam(\widehat{W})=r<\frac{1}{100}$, where $\widehat{W}$ as in
	\eqref{eq:0}. If $\widehat{W}\cap K=\emptyset$, we have nothing to prove. 
	We now suppose that $\widehat{W}\cap K\neq \emptyset$; we can find a point 
	$x_0\in S$, such that $\widehat{W}\subset B(x_0,r)$.

	We consider the Lipschitz function $\phi:\mathbb{R}\to [0,1]$ defined by
	\[
	\phi(t)=\begin{cases}
		0,&t\leq \frac{1}{4}\\
		4(t-\frac{1}{4}),&\frac{1}{4}<t\leq \frac{1}{2}\\
		1,&\frac{1}{2}<t\leq 2\\
		-4(t-2)+1,&2<t\leq \frac{9}{4}\\
		0,&t>\frac{9}{4}.
	\end{cases}
	\]

	We consider $\pi:\mathbb{R}^{n}\to \mathbb{R}^{n}$ defined by
	\[
	\pi(x)=\left( 1-\phi(\left\vert x \right\vert)
	\right)x+\phi(\left\vert x \right\vert)\frac{x}{\left\vert x \right\vert};
	\]
	when $\left\vert x \right\vert\leq \frac{1}{4}$ or $\left\vert x
	\right\vert\geq \frac{9}{4}$, $\pi(x)=x$; when $\frac{1}{2}\leq
	\left\vert x \right\vert\leq 2$, $\pi(x)=\frac{x}{\left\vert x
	\right\vert}$. Also $\pi$ is a Lipschitz map with
	\begin{equation}\label{eq:Lipschitzconstant}
		\Lip\left( \pi\vert_{B(x_0,r)} \right)\leq \frac{1}{1-r}.
	\end{equation}

	We put $\widetilde{\varphi}=\pi\circ\varphi_1$; then 
	\[
	\widetilde{\varphi}(\partial B(0,1)\cap \Omega)\subset \partial B(0,1)\cap
	\Omega.
	\]

	For $\varepsilon>0$ small, we consider the Lipschitz map $\psi_{\varepsilon}$
	defined by 
	\[
	\psi_{\varepsilon}(x)=\left( 1-\widetilde{\phi}_{\varepsilon}(\left\vert x \right\vert)
	\right)x+\widetilde{\phi}_{\varepsilon}(\left\vert x \right\vert)\left\vert x
	\right\vert\widetilde{\varphi}\left(\frac{x}{\left\vert x \right\vert}\right),
	\]
	where $\widetilde{\phi}_{\varepsilon}:\mathbb{R}\to [0,1]$ given by
	\[
	\widetilde{\phi}_{\varepsilon}(t)=\begin{cases}
		1,&t\leq 1\\
		-\frac{1}{\varepsilon}(t-1)+1,& 1<t\leq 1+\varepsilon\\
		0,&t\geq 1+\varepsilon.
	\end{cases}
	\]
	It is clear that $\psi_{\varepsilon}(x)=x$ for $\left\vert x \right\vert\geq
	1+\varepsilon$, $\psi_{\varepsilon}(x)=\left\vert x \right\vert
	\widetilde{\varphi}\left( \frac{x}{\left\vert x \right\vert} \right)$ for 
	$\left\vert x \right\vert\leq 1$.

	We consider the map $\widetilde{\pi}:\mathbb{R}^{3}\to\mathbb{R}^{3}$ given by
	\[
	\widetilde{\pi}(x)=\begin{cases}
		x,&0\leq \left\vert x \right\vert\leq 1\\
		\frac{x}{\left\vert x \right\vert}, &1<\left\vert x \right\vert\leq 2\\
		(2\left\vert x \right\vert-3)\frac{x}{\left\vert x \right\vert},&
		2<x\leq 3\\
		x,&\left\vert x \right\vert\geq 3,
	\end{cases}
	\]
	it is Lipschitz, thus the map
	$\widetilde{\psi_{\varepsilon}}:=\widetilde{\pi}\circ\psi_{\varepsilon}$ is
	also Lipschitz. It is easy to see that 
	$\widetilde{E}:=\widetilde{\psi}_{\varepsilon}(E)$ is a competitor of $E$,
	because that $\{ \varphi_t' \}_{0\leq t\leq 1}$ defined by 
	$\varphi_t'=(1-t)\id+t\widetilde{\psi_{\varepsilon}}$ is a deformation.
	We will compare $\widetilde{E}$ with $E$. Since $E$ is a cone,
	$\widetilde{\pi}(x)$ lie in the line through $0$ and $x$, $\widetilde{\pi}$
	is the radial projection into the sphere $\partial B(0,1)$ on the annulus 
	$1\leq \left\vert x \right\vert\leq 2$, and $\widetilde{\psi}_{\varepsilon}$ is
	identity out of the ball $B(0,1+\varepsilon)$, we can get that $\widetilde{E}$
	and $E$ coincide out of the ball $\overline{B(0,1)}$. Since $E$ is minimal,
	we have that 
	\begin{equation}\label{eq:comparemeasure}
		\H^{2}(E\cap \overline{B(0,1)})\leq \H^{2}(\widetilde{E}\cap
		\overline{B(0,1)}).
	\end{equation}
	Recall that $\widehat{W}\subset B(x_0,r)$; if we put 
	\begin{equation}\label{eq:assumption}
		R=\diam((W_1\cap K)\cup \widetilde{\varphi}(W_1\cap K)),
	\end{equation}
	where $W_1=\{ x\mid \varphi_1(x)\neq x \}$, then $R\leq r$, thus
	on the sphere $\partial B(0,1)$, we can see that $\widetilde{E}$
	and $E$ coincide out of $B(x_0,R)$, thus we can easily get that 
	\begin{equation}\label{eq:measureonsphere}
		\H^2(\widetilde{E}\cap \partial B(0,1))=\H^2(\widetilde{E}\cap \partial B(0,1)\cap
		B(x_0,R))\leq 4\pi R^2.
	\end{equation}
	Applying Theorem 3.2.22 in \cite{Federer:1969}, we have 
	\begin{equation}\label{eq:calculation}
		\begin{aligned}
			\H^2(E\cap \overline{B(0,1)})&=\H^2(E\cap B(0,1))+\H^2(E\cap \partial
			B(0,1))\\
			&=\int_{0}^{1}\H^1(E\cap \partial B(0,t))d t\\
			&=\left( \int_{0}^{1}tdt \right)\H^1(E\cap \partial B(0,1))\\
			&=\frac{1}{2}\H^1(E\cap \partial B(0,1)).
		\end{aligned}
	\end{equation}
	By the construction of $\widetilde{E}$, we know that $\widetilde{E}$ coincide with a
	cone in the ball $B(0,1)$, thus the same reason as above, we have 
	\[
	\H^2(\widetilde{E}\cap
	\overline{B(0,1)})=\frac{1}{2}\H^1(\widetilde{\varphi}(E\cap
	B(0,1)))+\H^2(\widetilde{E}\cap \partial B(0,1)).
	\]
	We combine this with \eqref{eq:comparemeasure}, \eqref{eq:measureonsphere}
	and \eqref{eq:calculation}, and get that 
	\[
	\H^1(K)\leq \H^1(\widetilde{\varphi}(K))+8\pi R^2.
	\]
	By our construction of $\widetilde{\varphi}$, we have that for any $x\in\partial
	B(0,1)$, if $\varphi_1(x)=x$, i.e., $x\not\in W_1$, then $\widetilde{\varphi}(x)=x$, 
	thus $K\setminus W_1=\widetilde{\varphi}(K\setminus W_1)$. 
	We get that
	\begin{equation}\label{eq:moreprecise}
		\begin{aligned}
			\H^1(K\cap W_1)&=\H^1(K)-\H^1(K\setminus W_1)\\
			&\leq \H^1(\widetilde{\varphi}(K))-\H^1(K\setminus W_1)+8\pi R^2\\
			&\leq \H^1(\widetilde{\varphi}(K\cap W_1))+8\pi R^2.
		\end{aligned}
	\end{equation}
	Since $R=\diam((W_1\cap K)\cup \widetilde{\varphi}(W_1\cap K))$,
	we can show that 
	\[
	\H^1(K\cap W_1)+\H^1(\widetilde{\varphi}(K\cap W_1))\geq R,
	\]
	combine this with \eqref{eq:moreprecise}, and get that 
	\[
	\H^1(K\cap W_1)\leq \frac{1+8\pi R}{1-8\pi R}\H^1(\widetilde{\varphi}(K\cap
	W_1))\leq (1+100R)\H^1(\widetilde{\varphi}(K\cap W_1)),
	\]
	and by \eqref{eq:assumption} and \eqref{eq:Lipschitzconstant}, 
	\[
	\H^1(K\cap W_1)\leq \frac{1+100R}{1-r}\H^1(\varphi_1(K\cap W_1))\leq
	(1+200r)\H^1(\varphi_1(K\cap W_1)),
	\]
	the result immediately follows.
\end{proof}

\begin{proposition}\label{prop:onsphere}
	Let $\Omega$, $L_{1}$ be as in \eqref{simple}. Let $E\subset \Omega$ be a 
	minimal cone, and set $K=E\cap \partial B(0,1)$. Then $K$ consists of arcs 
	$C_{i}$ of great circles. These arcs can only meet at their extremities.
	For each extremity, if it is not in $L_1$, then it is a common extremity of
	exactly three arcs which meet with $120^{\circ}$ angles.
\end{proposition}
A point in $K\setminus L_1$ is called to be a $Y$-point if it is a common extremity
of exactly three curves which meet with $120^{\circ}$ angles. 
\begin{proof}
	Applying Theorem \ref{thm:regularityonsphere} and Proposition
	\ref{prop:regularityonshpere},
	we can get $K$ consists of $C^{1,1/2}$ curves, these curves only
	meet at their extremities. For each extremity, if it is not in $L_1$, then
	it is a common extremity of exactly three curves which meet with
	$120^{\circ}$. For any point $x$ in the interior of such a curves $C_j$, by
	the same proof as in \cite[Proposition 14.1]{David:2009}, we can get that
	there is a neighborhood $U_x$ such that $U_x\cap C_j$ is an arc of great
	circles. From this, we can immediately deduce the result. 
\end{proof}

\begin{lemma}\label{le:BPC}
	Let $\Omega,L_1, E, K$ be as in the proposition above. For any $x\in L_1$, we
	denote by $\Omega_x$ the half plane through $0$ which is perpendicular to
	$L_1$ and the straight line joining $x$ and $0$. Then, for any point $x\in
	K\cap L_1$, any blow-up limit of $K$ at $x$ is a sliding minimal cone in
	$\Omega_x$ with sliding boundary $L_x=\Omega_x\cap L_1$.
\end{lemma}
The proof of this Lemma is almost the same as in the first part of the proof 
of Theorem 8.23 in \cite{David:2009}.
\begin{proof}
	Without loss of generality, we assume that $x=(1,0,0)$. Then $\Omega_x=\{
	(0,x_2,x_3)\mid x_2\in\mathbb{R},x_3\geq 0 \}$, $L_x=\{ (0,x_2,0)\mid x_2\in
	\mathbb{R}\}$. Let $r_k>0$, $r_k\to 0$. Suppose that 
	\[
	\frac{1}{r_k}(K-x)\to Z
	\]
	and 
	\[
	\frac{1}{r_k}(E-x)\to F.
	\]
	It is quite easy to see that $Z\subset \Omega_x$ and $Z\subset F$.
	Theorem 24.13 in \cite{David:2014} says that $F$ is a sliding minimal cone
	in $\Omega$ with sliding boundary $L_1$. We denote by $D$ the line though
	the points $0$ and $x$. As in \cite{David:2009}, page 140, we can get that 
	\[
	F=D\times F^{\sharp}\text{ where } F^{\sharp}=F\cap \Omega_x.
	\]
	Similarly, we can get $F^{\sharp}$ is a sliding minimal cone in $\Omega_x$
	with sliding boundary $L_x$. Let us check that $Z=F^{\sharp}$. It suffices to
	show that $F^{\sharp}\subset Z$, since we already know that 
	$Z\subset F^{\sharp}$. We take any $z\in F^{\sharp}$, then there exists a
	sequence $z_k\in E$ such that 
	\begin{equation}\label{eq:BPC5}
	\frac{z_k-x}{r_k}\to z.
	\end{equation}
	Since $z\in \Omega_x$, and $\Omega_x$ is perpendicular to the line $D$ which
	pass through the points $0$ and $x$, we get that the angles between the line
	$D$ and the segments which join the pints $x$ and $z_k$ tend to
	$\frac{\pi}{2}$, i.e. 
	\[
		\theta_k={\rm Angle}(z_k-x,D)\to \frac{\pi}{2}.
	\]
	If $z=0$, by \eqref{eq:BPC5}, we get that 
	\begin{equation}\label{eq:BPC6}
	\frac{\left\vert z_k\right\vert-1}{r_k}\to 0.
	\end{equation}
	If $z\neq 0$, we will show that
	\begin{equation}\label{eq:BPC7}
	\frac{\left\vert z_k\right\vert-1}{\left\vert
	z_k-x\right\vert}\to 0.
	\end{equation}
	We put $\gamma_k={\rm Angle}(z_k,x)$. Then $\gamma_k\to 0$. Since $z\neq 0$,
	we have that $|z_k-x|\neq 0$ and $\gamma_k\neq 0$ for $k$ large. We consider
	the triangle formed by the vertices $0$, $x$ and $z_k$. We get that 
	$|z_k-x|\geq |z_k|\sin\gamma_k$. Thus 
	\[
		\left\vert \frac{\cos \gamma_k -1}{|z_k-x|} \right\vert\leq
		\frac{1-\cos\gamma_k}{|z_k|\sin\gamma_k}\to 0.
	\]
	Since 
	\[
		\langle x,z_k-x \rangle=|x||z_k-x|\cos\theta_k=|z_k-x|\cos\theta_k
	\]
	and 
	\[
			\langle x,z_k-x \rangle=\langle x,z_k \rangle-|x|^2=|z_k|\cos\gamma_k-1,
	\]
	we get that 
	\[
		|z_k|\cos\gamma_k=1+|z_k-x|\cos\theta_k.
	\]
	Hence 
	\[
		\frac{|z_k|-1}{|z_k-x|}=\frac{\cos\theta_k}{\cos\gamma_k}+
		\frac{1-\cos\gamma_k}{|z_k-x|\cos\gamma_k}\to 0.
	\]

	In the case $z\neq 0$, we get, from \eqref{eq:BPC5} and \eqref{eq:BPC7}, that
	\begin{equation}\label{eq:BPC9}
	\frac{\left\vert z_k\right\vert-1}{r_k}\to 0.
	\end{equation}
	Thus, from \eqref{eq:BPC6} and \eqref{eq:BPC9}, we get that
	\[
	\frac{1}{r_k}\left( \frac{z_k}{\left\vert z_k\right\vert}
	-x\right)=\frac{1-\left\vert z_k\right\vert}{r_k}\cdot\frac{z_k}{\left\vert
	z_k\right\vert}+\frac{z_k-x}{r_k}\to z.
	\]
	But we know that $\frac{z_k}{\left\vert z_k\right\vert}\in K$, thus $z\in Z$.

\end{proof}

\begin{lemma}\label{lemma:boundaryofcone}
	Let $\Omega,L_1, E$ be as in the proposition above, $S=K\setminus L_1$. 
	For any $x\in \overline{S}\cap L_1$, there is a radius $r>0$ such that there 
	is no $Y$-point in $S\cap B(x,r)$. Moreover, if a blow-up limit of $K$ at 
	$x$ is a cone $V_{\beta,0}$ for some $\beta\in (0,\frac{\pi}{6}]$, then 
	$K\cap B(x,r)$ is a union of two arcs of great circles meeting at $x$;
	in the other cases, $S\cap B(x,r)$ is an arc of great circle which perpendicular
	to $L_1$.
\end{lemma}
\begin{proof}
	We will prove that there ais only finite number of $Y$ points in $S$.
	We denote	$\mathbb{S}^{+}=\{ (x_1,x_2,x_3)\in \mathbb{S}^2 \mid x_3>0\}$. 
	Let $A$ be a connected component of $\mathbb{S}^{+}\setminus S$,
	$\overline{A}$ be the closure of $A$
	
	If $\overline{A}\cap L_1=\emptyset$, then $\overline{A}$ is a convex. Indeed,
	we get, from Proposition \ref{prop:onsphere}, that each $Y$-point in $S$ 
	must connect three arcs, these three arcs meet with $120^{\circ}$. Thus at
	each corner of $\partial A$, the interior angle of $\partial A$ at this point 
	must have be $120^{\circ}$, and $\overline{A}$ must be convex. 
	
	Now, if $\overline{A}\cap L_1\neq \emptyset$, $\overline{A}$ is also convex. 
	For the same reason, if the vertex of a corner of $\partial A$ is contained
	in $\mathbb{S}^{+}$, then the interior angle of $\partial A$ at this point 
	must have be $120^{\circ}$. If the vertex of a corner of $\partial A$ is
	contained in $L_1$, then the the interior angle of $\partial A$ at this
	point is no more than $180^{\circ}$. Thus $\overline{A}$ must be convex.

	The number corners in
	$\partial A$ must be finite. Indeed, there are at most four corners which 
	touch the boundary $L_1$, because $A$ is convex. If there are infintely many
	corners in $\partial A$, then we can very easily to find $8$ corners, saying
	at points $B_1,B_2,\ldots, B_8$, such that these $8$ points are
	contained in $S$, and the geodesic connecting $B_i$ and $B_{i+1}$ is contained in
	$\partial A$, $i=1,2,\ldots,7$. We now consider the convex spherical polygon
	$B_1B_2\cdots B_8$. By using the Gauss-Bonnet theorem, for example see 
	\cite[Theorem \uppercase\expandafter{\romannumeral5}.2.7]{Chavel:2006}, we get that 
	\begin{equation}\label{eq:BC2}
		\alpha_1+\alpha_2+\frac{\pi}{3}\times 6+{\rm Area}(\overline{A})=2\pi,
	\end{equation}
	where $\alpha_1$ and $\alpha_2$ are the exterior angle of the
	corners of $\partial \overline{A}$ at point $B_1$ and $B_8$ respectivly. But
	that is impossible, the equation \eqref{eq:BC2} gives an absurdity. 

	If $\overline{A}$ are contained in $\mathbb{S}^{+}$, we assume that
	$\partial A$ has $n$ cornes, then Gauss-Bonnet theorem says that 
	\[
	\frac{n\pi}{3}+{\rm Area}(\overline{A})=2\pi,
	\]
	thus $n<6$, and ${\rm Area}\geq \frac{\pi}{3}$. Since the totall area of
	$\Omega\cap \partial B(0,1)$ is $\pi$, there are at most $6$ such connected
	components. Thus there is only a finite number of $Y$-point in
	$S$; otherwise, it should be infintely many connected component of 
	$\mathbb{S}^{+}\setminus S$ such that its corners does not touch $L_1$.

	Since there is only a finite number of $Y$-point in
	$S$, we get that for any $x\in \overline{S}\cap L_1$, there is a
	radius $r_x>0$ such that there is no $Y$-point in $S\cap B(x,r_x)$.

	Since $K$ is sliding almost minimal, any blow-up limit of $K$ at $x$ is a
	sliding minimal cone, denote by $Z$, and 
	\[
	\theta_K(x)=\H^1(Z\cap B(0,1)).
	\]
	If $Z$ is a cone like $V_{\beta,0}$ for some $\beta\in (0,\frac{\pi}{6}]$,
	then $K\cap B(0,r_x)$ must be two arcs, each of these two arcs is a part of 
	a great cicle, and these two arcs meet at $x$ with angle $\pi-2\beta$.
	If $Z$ is a half line perpendicular to $L_1$, then $K\cap B(0,r_x)$ is an
	arc which is a part of a great cicle, perpendicular to $L_1$ and through
	$x$. If $Z$ is the union of a line in $L_1$ and a half line which is
	perpendicular to $L_1$, then $K$ is the union of $B(0,r_x)\cap \{
	(x_1,x_2,0)\mid x_1^2+x_2^2=1 \}$ and an arc which is a part of a great
	cicle, perpendicular to $L_1$ and through $x$.
\end{proof}

\begin{lemma}\label{le:PPC}
	Let $\Omega$, $L_1$ be as in \eqref{simple}. Let $E\subset \Omega$ be a 
	sliding minimal cone, $K=E\cap\partial B(0,1)$, $S=K\setminus L_1$. Suppose
	that for each $x\in \overline{S}\cap L_1$, there is a radius $r>0$ such that
	$B(x,r)\cap S$ is an arc of a great circle which is perpendicular to $L_1$.
	Then there are only there possible kinds of $S$, that is,
	$\overline{S}=Z\cap \partial B(0,1)$, where $Z$ is a 
	sliding minimal cone of type of one of $\mathbb{P}_{+}$, 
	$\mathbb{Y}_{+}$ and $\mathbb{T}_{+}$. And hence, $E=Z$ or
	$E=Z'\cup L_1$ where $Z'$ is a cone of type $\mathbb{P}_{+}$ or
	$\mathbb{Y}_{+}$. 
\end{lemma}
\begin{proof}
	We put $\mathbb{S}^{+}=\Omega\cap \partial B(0,1) \setminus L_1$.  
	Let $A$ be a connected component of $\mathbb{S}^{+}\setminus K$, $\overline{A}$
	be the closure of $A$.  By Proposition \ref{prop:onsphere} and
	Lemma \ref{lemma:boundaryofcone}, the boundary of $\overline{A}$ is a 
	spherical polygon whose sides are geodesics of the unit sphere.
	Using the Gauss-Bonnet theorem, see \cite[Theorem
	\uppercase\expandafter{\romannumeral5}.2.7]{Chavel:2006}, we get that 
	\begin{equation}\label{eq:angles}
		\alpha_1+\alpha_2+\cdots+\alpha_n+{\rm Area}(\overline{A})=2\pi,
	\end{equation}
	where $\alpha_1,\alpha_2,\cdots,\alpha_n$ are the exterior angle of the
	corners of $\partial \overline{A}$. From Lemma \ref{lemma:boundaryofcone}
	and Proposition \ref{prop:onsphere}, we can see that, if a corner touch 
	$L_1$, then, in the situation of this lemma, the corresponding exterior 
	angle must be $\frac{\pi}{2}$; if a corner do not touch $L_1$, the 
	corresponding exterior angle must be always $\frac{\pi}{3}$. It is quite 
	clear that there are at least two corners on $\partial \overline{A}$, and it 
	cannot happen that there only one corner touching $L_1$. 

	We now consider the equation \eqref{eq:angles}. If $n=2$, then the two
	corners must touch $L_1$, thus $\alpha_1=\alpha_2=\frac{\pi}{2}$, and
	${\rm Aera}(\overline{A})=\pi$. $\overline{A}$ is a quarter of unity sphere.

	Let us split $n=3$ into two cases. If there is no corner touching
	$L_1$, then $\alpha_1=\alpha_2=\alpha_3=\frac{\pi}{3}$, and $\overline{A}$
	is an equilateral polar triangle with ${\rm Area}(A)=\pi$. If it has corner
	on $\partial \overline{A}$, then there are at least two, thus
	$\alpha_1=\alpha_2=\frac{\pi}{2}$, $\alpha_3=\frac{\pi}{3}$, and
	$\overline{A}$ is a isosceles polar triangle with 
	${\rm Area}(A)=\frac{2\pi}{3}$. More precisely, the base of $\overline{A}$
	is an arc contained in $L_1\cap \partial B(0,1)$ with length $\frac{2\pi}{3}$, the 
	vertex opposite to the base is the point $(0,0,1)$.

	Similarly, for $n=4$, we can get two kinds of spherical quadrilaterals, one
	is spherical quadrilaterals with equal angles $\frac{2\pi}{3}$ and with area
	$\frac{2\pi}{3}$, another one is spherical quadrilaterals with one side 
	contained in $L_1\cap \partial B(0,1)$ and with area $\frac{\pi}{3}$.

	We can easily see that $n$ can not be larger than $5$; otherwise, we can
	deduce from \eqref{eq:angles} that ${\rm Area}(\overline{A})\leq 0$, which is
	impossible. For the same reason, when $n=5$, there is only one kind of spherical
	pentagons. That is, a spherical with all of corners are contained in
	$\mathbb{S}^{+}$ and with area $\frac{\pi}{3}$.

	Since each connected component of $\mathbb{S}^{+}\cap \partial B(0,1)\setminus K$ 
	has at least area $\frac{\pi}{3}$, and the total area of $\Omega\cap \partial
	B(0,1)$ is $2\pi$, there are at most six connected component. If there is no
	$Y$ point on $\mathbb{S}^{+}$, $E\cap \mathbb{S}^{+}$ must be a half
	circle which is contained in $\Omega\cap \partial B(0,1)$ and perpendicular
	to $L_1$. Thus $E=Z$ or $E=L_1\cup Z$, where $Z$ is a cone of type
	$\mathbb{P}_{+}$, hence $\overline{S}=Z\cap \partial B(0,1)$.

	If there is only one $Y$ point on $\mathbb{S}^{+}$,
	then each connect component of $\mathbb{S}^{+}\cap \partial B(0,1)\setminus K$
	must be a polar triangle with base contained in $L_1\cap \partial B(0,1)$. 
	By our discussion for $n=3$, we get that each such connected component is an
	isosceles polar triangle with area $\frac{2\pi}{3}$. Thus this $Y$ point
	must be $(0,0,1)$, and $E=Z$ or $E=L_1\cup Z$, where $Z$ is a cone of type
	$\mathbb{Y}_{+}$, hence $\overline{S}=Z\cap \partial B(0,1)$.

	If there are two $Y$ points on $\mathbb{S}^{+}$,
	then there are at least two polar triangles with base contained in 
	$L_1\cap \partial B(0,1)$, and the vertices opposite to the bases must be 
	the point $(0,0,1)$, that is impossible.

	If there are three $Y$ points on
	$\mathbb{S}^{+}$, then these three points must be the vertices of a polar 
	triangle which is contained in $\mathbb{S}^{+}$, and this triangle is an
	equilateral polar triangle with area $\pi$. Each side of this triangle is a
	side of spherical quadrilateral, and the opposite side ot this spherical
	quadrilateral is contained in $L_1\cap \partial B(0,1)$. In this case
	$E=Z$ or $E=L_1\cup Z$, where $Z$ is a cone of type $\mathbb{T}_{+}$,
	thus $\overline{S}=Z\cap \partial B(0,1)$.

	Since each polar triangle contained in $\mathbb{S}^{+}$ has area
	$\pi$, there only one such triangle. If there are four $Y$ point in
	$\mathbb{S}^{+}$, then these four point in $E\cap \mathbb{S}^{+}$ must
	form a spherical quadrilateral, and $\mathbb{S}^{+}\setminus K$ consists of
	five region, that is, five connect component, each of them is a spherical
	quadrilateral, one of them is contained in $\mathbb{S}^{+}$, and each of the
	rest quadrilateral must has one side contained in $L_1$. Without loss of
	generality, we suppose that $(1,0,0)\in \overline{S}$, then those four
	$Y$ points are $(\frac{\sqrt{6}}{3},0,\frac{\sqrt{3}}{3})$,
	$(0,\frac{\sqrt{6}}{3},\frac{\sqrt{3}}{3})$,
	$(-\frac{\sqrt{6}}{3},0,\frac{\sqrt{3}}{3})$ and 
	$(0,-\frac{\sqrt{6}}{3},\frac{\sqrt{3}}{3})$, we denote them by $B_1'$,
	$B_2'$, $B_3'$ and $B_4'$ respectively. In this case, we will show
	that is impossible. We put $B_1=(\frac{\sqrt{6}}{3},0,0)$,
	$B_2=(0,\frac{\sqrt{6}}{3},0)$, $B_3=(-\frac{\sqrt{6}}{3},0,0)$ and 
	$B_4=(0,-\frac{\sqrt{6}}{3},0)$, and denote by $C$ the cube
	$B_1B_2B_3B_4B_1'B_2'B_3'B_4'$, by $\Gamma$ the union of the edges of cube 
	$C$. We denote by $F_0$ the face $B_1'B_2'B_3'B_4'$ of the cube $C$, denote by
	$F_1$ and $F_2$ the rectangle $B_1B_3B_3'B_1'$ and $B_2B_4B_4'B_2'$ 
	respectively. 
\begin{figure}[!htb]
	\centering
	\includegraphics[width=3in]{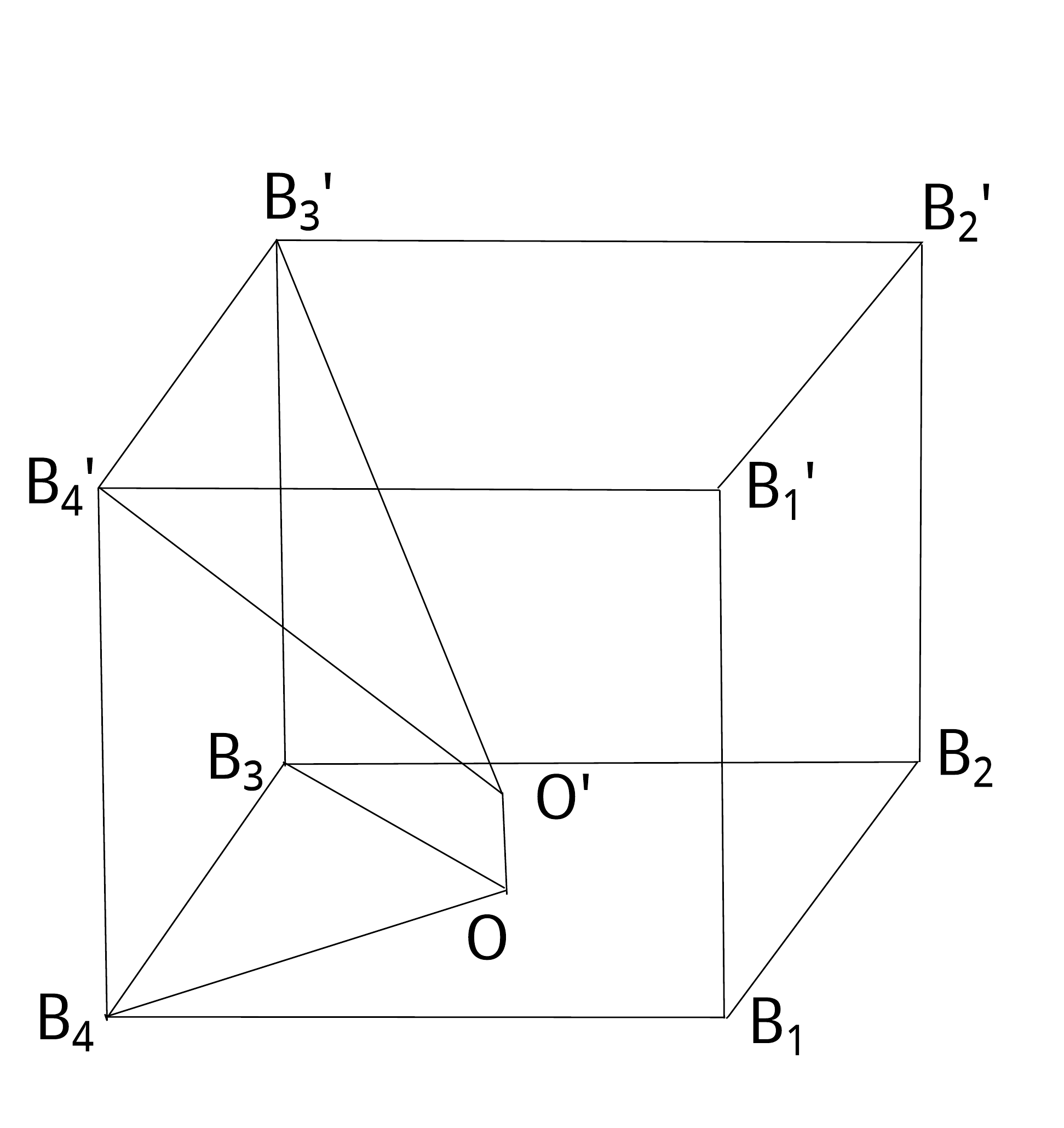}
	\label{fig:polytope}
\end{figure}
	
	We denote by $O$ and
	$O'$ the points $(0,0,0)$ and $(0,0,\frac{\sqrt{3}-\sqrt{2}}{3})$ respectively,
	and denote by $Q_1$, $Q_1$, $Q_3$ and $Q_4$ the polyhedrons 
	$B_1B_2OB_1'B_2'O'$, $B_2B_3OB_2'B_3'O'$, $B_3B_4OB_3'B_4'O'$ and
	$B_4B_1OB_4'B_1'O'$ respectively. We now take points $x_1$, $x_2$, $x_3$ and
	$x_4$ in the interior of pyramids $OB_1B_2B_2'B_1'$, $OB_2B_3B_3'B_2'$,
	$OB_3B_4B_4'B_3'$ and $OB_4B_1B_1'B_4'$ respectively. 

	Let maps $\Pi_{Q_1,x_1}$, $\Pi_{Q_2,x_2}$, $\Pi_{Q_3,x_3}$ and
	$\Pi_{Q_4,x_4}$ be as in \eqref{eq:radialprojection}. We take
	$\psi=\Pi_{Q_4,x_4}\circ\Pi_{Q_3,x_3}\circ\Pi_{Q_2,x_2}\circ\Pi_{Q_1,x_1}$.
	Then $\psi$ is a Lipschitz map. We now consider the competitor 
	$\tilde{E}=\psi(E)$ of $E$.
	
	We denote $B_5=B_1$. By a simple calculation, we can get that 
	\begin{align*}
		\H^2(E\setminus\tilde{E})-\H^2(\tilde{E}\setminus E)&=
		\sum_{i=1}^{4}|OB_i'B_{i+1}'|-\left(\sum_{i=1}^{4}
		O'B_i'B_{i+1}'+\sum_{i=1}^{4}|OO'B_i'| \right)\\
				&= \frac{2\sqrt{2}+4\sqrt{3}-4}{3}>0,
	\end{align*}
	that is contradict to minimality of $E$.
	
	It could not happen that there are at least six $Y$ point in $\mathbb{S}^{+}$,
	because otherwise, there will be at least six spherical quadrilaterals which
	touch $L_1$, but we know that such a quadrilateral has area
	$\frac{\pi}{3}$, and total area of $\mathbb{S}^{+}$ is $2\pi$, that is
	impossible. If there are five $Y$ point in $\mathbb{S}^{+}$, similar to
	above case, these five point form a spherical pentagon. By the same
	techniques used for the above case, we can prove that this is impossible. 
\end{proof}
\begin{lemma}\label{le:PPCV}
	Let $\Omega, L_1$ be as in \eqref{simple}. Let $E\subset \Omega$ be a 
	sliding minimal cone, $K=E\cap\partial B(0,1)$. If there
	exists a point $x_0\in K\cap L_1$ such that a blow-up limit of
	$K$ at $x_0$ is a sliding minimal cone $V_{\alpha,0}$ for some $\alpha\in
(0,\frac{\pi}{6}]$, then $E$ is a cone of type $\mathbb{V}$.
\end{lemma}
\begin{proof}
	By Lemma \ref{lemma:boundaryofcone}, there exists a radius $r>0$ such that 
	$S\cap B(x_0,r)$ is a union of two arcs. That is, $E\cap
	B(x_0,r)=Z\cap B(x_0,r)$ where $Z$ is a cone of type $\mathbb{V}$. Without
	loss of generality, we assume $x_0=(1,0,0)$.

	We denote $\mathbb{S}^{+}=\Omega\cap \partial B(0,1) \setminus L_1$.
	Let $A$ be a connected component of $\mathbb{S}^{+}\setminus K$ which contains
	a corner with interior angle $\alpha$. Using the Gauss-Bonnet Theorem, see
	\cite[Theorem \uppercase\expandafter{\romannumeral5}.2.7]{Chavel:2006},
	we get that 
	\begin{equation}\label{eq:angles2}
		\alpha_1+\alpha_2+\cdots+\alpha_n+{\rm Area}(\overline{A})=2\pi,
	\end{equation}
	where $\alpha_1,\alpha_2,\cdots,\alpha_n$ are the exterior angle of the
	corners of $\partial \overline{A}$. We know that there are two corners which 
	touch the boundary $L_1$, assume that $\alpha_1=\pi-\alpha$ and $\alpha_2$ 
	are the corresponding exterior angles. It is quite easy to see that $A$ is
	contained in a spherical lune enclosed by two great circles with angle $\alpha$,
	thus ${\rm Area}(\overline{A})\leq 2\alpha$.

	If $\alpha_2=\frac{\pi}{2}$, then ${\rm Area}(\overline{A})<2\alpha$,
	\[
	(\pi-\alpha)+\frac{\pi}{2}+\frac{\pi}{3}\times (n-2)+{\rm
	Area}(\overline{A})=2\pi,
	\]
	thus
	\[
	\frac{\pi}{2}-\alpha< \frac{n-2}{3}\pi<\frac{\pi}{2}+\alpha.
	\]
	Since $\alpha\in (0,\frac{\pi}{6}]$, we get that $3< n<4$, that is impposible.

	If $\alpha_2\neq \frac{\pi}{2}$, then $\alpha_2\geq \frac{5\pi}{6}$, thus 
	\[
	2\pi=(\pi-\alpha)+\alpha_2+\frac{(n-2)\pi}{3}+{\rm
	Area}(\overline{A})> (\pi-\alpha)+\frac{5\pi}{6}+\frac{(n-2)\pi}{3}.
	\]
	Since $\alpha\leq \frac{\pi}{6}$, we get that $2\leq n< 3$, hence $n=2$. In
	this case, $A$ must be a spherical lune enclosed by two great circles with angle
	$\alpha$, and $E=\mathbb{R}\times V_{\alpha,0}$.
\end{proof}
\begin{theorem}\label{thm:MCCB}
	Let $\Omega$, $L_{1}$ be as in \eqref{simple}. Let $E\subset \Omega$ be a 
	sliding minimal cone. If $L_1\subset E$ and $E\setminus L_1\neq \emptyset$. 
	Then $E=Z\cup L_1$, $Z$ is a cone of type $\mathbb{P}_{+}$ or $\mathbb{Y}_{+}$. 
\end{theorem}
\begin{proof}
	The result immediately follows from Lemma \ref{lemma:boundaryofcone}, Lemma
	\ref{le:PPC} and Lemma \ref{le:PPCV}. Indeed, by putting $K=E\cap \partial
	B(0,1)$ and $S=K\setminus L_1$, Lemma \ref{le:BPC} says that any blow-up
	limit of $K$ at a point $x\in K\cap L_1$ is a one dimensional sliding 
	minimal cone. By Lemma \ref{listofonedimensionalminimalcone}, there exist
	only there possible cases for such a minimal cone. That is, a half line
	$P_0$ perpendicular to $L_1$, or $P_0$ union the line which pass though $x$
	perpendicular the segment $[0,x]$ and is contained in $L_1$, or a cone
	$V_{\alpha,0}$. If it is a cone $V_{\alpha,0}$, then by Lemma \ref{le:PPCV},
	$E=\mathbb{R}\times V_{\alpha,0}$, which is impossible. For any
	$x\in \overline{S}\cap L_1$, by Lemma \ref{lemma:boundaryofcone}, there 
	exists a radius $r>0$ such that $S\cap B(x,r)$ is an arc of a great cicle 
	which is perpendicular to $L_1$, and by Lemma \ref{le:PPC}, $E=Z\cup L_1$,
	where $Z$ is a cone of type of one of $\mathbb{P}_{+}$,
	$\mathbb{Y}_{+}$ and $\mathbb{T}_{+}$, but for the last case,
	it is impossible, because we know that $Z\cup L_1$ is not minimal when
	$Z$ is of type $\mathbb{T}_{+}$. We get that $E=Z\cup L_1$, where
	$Z$ is a cone of type $\mathbb{P}_{+}$ or $\mathbb{Y}_{+}$.
\end{proof}
\begin{remark}\label{re:listofminimalcones}
	We claim that the list of sliding minimal cones is the following: cones of 
	type $\mathbb{P}_{+}$, cones of type $\mathbb{Y}_{+}$, the plane $L_1$, 
	cones like $\mathscr{R}(\mathbb{R}\times V_{\alpha,0})$ with $0<\alpha\leq \frac{\pi}{6}$, 
	cones $L_1\cup Z$ where $Z$ are cones of type $\mathbb{P}_{+}$ or
	$\mathbb{Y}_{+}$, and cones type $\mathbb{T}_{+}$. 
\end{remark}
We did not prove that a cone of type $\mathbb{T}_{+}$ is sliding minimal.
Indeed it can probably be proved by using calibration, but this may take us 
too much time, we do not want to do it here. It is also not too hard to check
that a cone like $\mathbb{R}\times V_{\alpha,0}$ is sliding minimal if and only 
if $0<\alpha\leq \frac{\pi}{6}$. One of possible ways to do this is to sue 
almost the same technique as in Lemma \ref{le:PLMC}, but again we omit it. In 
fact, we do not need know whether or not a cone of type $\mathbb{T}_{+}$ or 
like $\mathbb{R}\times V_{\alpha,0}$ is minimal in this paper. For the rest in
the list, we know from Lemma \ref{lemma:boundaryofcone}, Lemma \ref{le:PPC} and 
Lemma \ref{le:PPCV} that they are sliding minimal.

\section{Reifenberg's theorem}
We want to use a result, Theorem 2.2 in \cite{DDT:2008}. But here 
we are in the half space, the theorem can not be used directly, it 
should be adapted a little bit. 

Let $n$ and $d$ be two integers with $2\leq d<n$. We take 
\begin{equation}\label{eq:ndimhalf}
	\begin{gathered}
		\Omega_n=\{ (x_1,x_2,\cdots,x_n)\in\mathbb{R}^{n}\mid x_n\geq 0 \},\\
		L_1=\{ (x_1,x_2,\cdots,x_n)\in\mathbb{R}^{n}\mid x_n=0 \}.\\
	\end{gathered}
\end{equation}
We let $\sigma$ be the reflection with respect to $L_1$. That is, the function
$\mathbb{R}^{n}\to\mathbb{R}^{n}$ defined by
\[
	\sigma(x_1,\cdots,x_{n-1},x_n)=(x_1,\cdots,x_{n-1},-x_n).
\]

Let $TG$ be a class of sets defined as in \cite[p.6]{DDT:2008}, which consists
of 3 kinds of cones (centered at any point in $\mathbb{R}^{n}$) of dimension 
$d$ in $\mathbb{R}^{n}$. In particular, if $n=3$, $d=2$,
$TG$ consists of planes, cones which are the union of three half planes bounded 
by a line while the angle between any two half planes is larger than a constant
$\tau_0>0$ (they look like cones of type $\mathbb{Y}$), and cones which are
union of several faces that meet only by sets of three and with angles between
two adjacent faces and angles between the spines larger than a constant 
$\tau_0>0$ (cones of type $\mathbb{T}$ and cones look like of type $\mathbb{T}$ 
are such cones; cone $Z\cup \sigma(Z)$ is also a such cone, where $Z$ is a cone 
of type $\mathbb{Y}_{+}$ or $\mathbb{T}_{+}$. Of course, far more than these cones).

For any $x\in\mathbb{R}^{n}$, $r>0$, we will consider $d_{x,r}$ a variant of
the Hausdorff distance on closed sets, which is defined by 
\[
d_{x,r}(E,F)=\frac{1}{r}\max\left\{ \sup_{z\in E\cap
B(x,r)}\dist(z,F),\sup_{z\in F\cap B(x,r)}\dist(z,E) \right\}.
\]
If $E$ and $F$ are two cones centered at $x$, then $d_{x,r}(E,F)=d_{x,1}(E,F)$
for any $r>0$.
\begin{theorem}
	Let $E\subset\mathbb{R}^{n}$ be a compact set that contains origin with
	$\sigma(E)=E$, and suppose that for each $x\in E\cap B(0,2)$ and each ball 
	$B(x,r)\subset B(0,2)$, we can find $Z(x,r)\in TG$ that contains $x$, such that
	\[
	d_{x,r}(E,Z(x,r))\leq \varepsilon.
	\]
	Suppose, additionally, that $Z(\sigma(x),r)=\sigma(Z(x,r))$. 
	If $\varepsilon>0$ is small enough, depending only on $n$, $d$ and $\tau_0$,
	then there exist a cone $Z\in TG$ centered at origin and a mapping 
	$f:B(0,3/2)\to B(0,2)$ with the following properties:
	\begin{equation}\label{eq:PT1}
		\begin{gathered}
			\sigma(Z)=Z,\ f\circ\sigma=\sigma\circ f,\ \| f-\id \|_{\infty}\leq \alpha,\\
			(1+\alpha)^{-1}\left\vert x-y \right\vert^{1+\alpha}\leq \left\vert
			f(x)-f(y) \right\vert\leq (1+\alpha)\left\vert x-y
			\right\vert^{\frac{1}{1+\alpha}},\\
			B(0,1)\subset f\left( B\left( 0,\frac{3}{2} \right) \right)\subset
			B(0,2),\\
			E\cap B(0,1)\subset f\left(Z\cap B\left( 0,\frac{3}{2}
			\right)\right)\subset E\cap B(0,2),
		\end{gathered}
	\end{equation}
	where $\alpha$ only depends on only $n$, $d$, $\tau_0$ and $\varepsilon$,
	and $\tau_0$ is defined as in \cite[(2.7) and (2,8)]{DDT:2008}.
\end{theorem}
\begin{proof}
	The proof is essential the same as in \cite{DDT:2008}, we only need do a
	little change. Here we use same notation as in \cite{DDT:2008}. We firstly
	remark that $\sigma(E_i)=E_i$, $i=1,2$ or $3$, where $E_i$ are defined
	as in \cite[p.11, p.12]{DDT:2008}. 

	Next, we modify a little the construction of a good covering of $E$ at
	scale $2^{-n}$, that is in Section 5 in 
	\cite[Covering and partition of unity]{DDT:2008}. The first step is just same the
	as the original construction; if the condition (4.36) in \cite{DDT:2008} holds,
	we cover $E_3\cap B(0,199/100)=\{ 0 \}$ with the ball $B_{i_0}=B(0,2^{-n-20})$, 
	and set $I_3=\{ i_0 \}$; if the condition (4.35) in \cite{DDT:2008} holds, 
	we take $I_{3}=\emptyset$ and choose no ball. In the second step, for the 
	construction of a covering of 
	\[
	E_{2}'=E_2\cap B(0,198/100)\setminus \frac{7}{4}B_{i_0},
	\]
	we modify a little the original construction to adapt to our case. We put
	$E_2''=E_2'\cap \Omega_n$, then $E_2'=E_2''\cap \sigma(E_2'')$. Select a
	maximal subset $X_2''$ of $E_2''$, with the property that different
	points of $X_2''$ have distances at least $2^{-n-40}$. We put $X_2=X_2''\cup
	\sigma(X_2'')$, and for accounting reasons, we suppose that 
	$X_2''=\{ x_i \}_{i\in I_2''}$, $I_2''\cap I_3=\emptyset$, and that
	$X_2'=\sigma(X_2'')=\{ x_i \}_{i\in I_2'}$, $I_2'\cap (I_2''\cup
	I_3)=\emptyset$. Let $I_2=I_2'\cup I_2''$, $X_2=X_2'\cup X_2''$. We put
	$r_i=2^{-n-40}$ and $B_i=B(x_i,r_i)$ for $i\in I_2$. We can see that the 
	balls $\overline{B_i}$, $i\in I_2$, cover $E_2'$. In the third step, we put 
	\[
	V_2=\frac{15}{8}B_{i_0}\cup \bigcup_{i\in I_2}\frac{7}{4}B_i \text{ and }
	E'=E_1\cap B\left( 0,\frac{197}{100} \right)\setminus V_2.
	\]
	Similarly to the above step, put $E_1''=E_1'\cap \Omega_n$, and select a
	maximal subset $X_1''$ of $E_1''$, with the property that different points
	of $X_1''$ have distances at least $2^{-n-60}$, and then suppose that 
	$X_1''=\{ x_i \}_{i\in I_1''}$ with $I_1\cap (I_2\cup I_3)=\emptyset$, and 
	that $X_1'=\sigma(X_1'')=\{ x_i \}_{i\in I_1'}$ with $I_1'\cap (I_1''\cup
	I_2\cup I_3)$. Set $I_1=I_1'\cup I_1''$,
	and $B_i=B(x_i,2^{-n-60})$ for $i\in I_1$. It is
	very easy to see that the balls $\overline{B_i}$, $i\in I_1$, cover
	$E_{1}'$. For the fourth and last step of the construction of the
	covering, we put 
	\[
	V_1=\frac{31}{16}B_{i_0}\cup \bigcup_{i\in I_2}\frac{15}{8}B_i\cup
	\bigcup_{i\in I_1}\frac{7}{4}B_i \text{ and } E_0'=\mathbb{R}^{3}\setminus
	V_1.
	\]
	We put $E_0''=E_0'\cap \Omega_n$, and pick a maximal subset $X_0''$ of
	$E_0''$, such that different points of $X_0''$ have distance at least
	$2^{-n-80}$, and then suppose that $X_0''=\{ x_i \}_{i\in I_0''}$ with 
	$I_0''\cap (I_1\cup I_2\cup I_3)=\emptyset$, and that $X_0'=\{ x_i \}_{i\in
	I_0'}$ with $I_0'\cap (I_0''\cup I_1\cup I_2\cup I_3)=\emptyset$.
	Set $I_0=I_0'\cup I_0''$, and $B_{i}=B(x_i,2^{-n-80})$ for $i\in I_0$, then
	the balls $\overline{B_i}$, $i\in I_0$, cover $E_0'$.

	For the selection of a partition of unity in equation (5.10) in
	\cite{DDT:2008}, we choose the $\widetilde{\theta}_i$ as the translation and
	dilation of a same model $\theta$, where $\theta$ is a smooth function such
	that $\theta(x)=1$ in $B(0,2)$, $\theta(x)=0$ out of $B(0,3)$, $0\leq
	\theta(x)\leq 1$ everywhere, and $\sigma\circ\theta=\theta\circ\sigma$. The
	rest of proof will be the same as in \cite{DDT:2008}.

	We now verify that $\sigma\circ f^{\ast}=f^{\ast}\circ\sigma$. It is clear 
	that $\sigma\circ f_0^{\ast}=f_0^{\ast}\circ \sigma$, 
	$\sigma\circ f_0=f_0\circ\sigma$,
	$\sigma\circ\psi_{i_0}^{\ast}=\psi_{i_0}^{\ast}\circ\sigma$, and
	$\sigma\circ\psi_{i_0}=\psi_{i_0}\circ\sigma$. By our construction of $X_0$,
	$X_1$ and $X_2$, we can see that 
	\begin{align*}
		g_n^{\ast}(x)&=\sum_{i\in I_n}\theta_i(x)\psi_i^{\ast}(x)\\
		&=\theta_{i_0}(x)\psi_{i_0}^{\ast}(x)+
		\sum_{i\in I_0'\cup I_1'\cup I_2'}\Big(\theta_i(x)\psi_i^{\ast}(x)+
		\theta_i(\sigma(x))\psi_i^{\ast}(\sigma(x))\Big),\\
	\end{align*}
	thus $\sigma\circ g_n^{\ast}=g_n^{\ast}\circ\sigma$. By induction on $n$, we
	can get that $\sigma\circ f_n^{\ast}=f_n^{\ast}\circ\sigma$ for all $n\geq 0$.
	$f^{\ast}$ is the limit of the sequence $f_n^{\ast}$, thus $\sigma\circ
	f^{\ast}=f^{\ast}\circ \sigma$. 
	
	Finally, by the same proof as above, we can prove that
	$\sigma\circ f= f\circ \sigma$.
\end{proof}

\begin{corollary}\label{co:PT}
	For each small $\tau>0$, we can find $\varepsilon>0$, that depends only on
	$n$, $\tau$ and $\tau_0$ such that if $E\subset \Omega$ is a closed set,
	$0\in E$ and $r>0$ are such that for $y\in E\cap B(0,3r)$ and $0<t\leq 3r$, we
	can find $Z(y,t)$, which is a minimal cone in $\mathbb{R}^{3}$ 
	when $0<t<\dist(y, L_1)$ and a sliding minimal cone in $\Omega$ with boundary 
	$L_1$ when $\dist(y,L_1)\leq t\leq 3r$, such that 
	\[
	d_{y,t}(E, Z(y,t))\leq \varepsilon,
	\]
	and in addition $Z(0,3r)$ is sliding minimal cone centered at $0$, then there
	is a sliding minimal cone centered at origin and a mapping 
	$f:B(0,3r/2)\cap \Omega\to B(0,2r)\cap\Omega$ with the following properties:
	\begin{equation}\label{eq:PT2}
		\begin{gathered}
			f(x)\in L_1\text{ for }x\in L_1\text{ and } \| f-\id \|_{\infty}\leq
			\tau,\text{ and}\\
			(1+\tau)^{-1}\left\vert x-y \right\vert^{1+\tau}\leq \left\vert
			f(x)-f(y) \right\vert\leq (1+\tau)\left\vert x-y
			\right\vert^{\frac{1}{1+\tau}},\\
			B(0,r)\cap\Omega\subset f\left( B\left( 0,\frac{3r}{2} \right)\cap
			\Omega \right)\subset B(0,2r)\cap \Omega,\\
			E\cap B(0,r)\subset f\left(Z\cap B\left( 0,\frac{3r}{2}
			\right)\right)\subset E\cap B(0,2r).
		\end{gathered}
	\end{equation}

\end{corollary}

\section{Regularity of sliding almost minimal sets
\uppercase\expandafter{\romannumeral1}}

In this section, we restrict ourselves to the half space $\Omega$, and prove
some boundary regularity for sliding almost minimal sets. 

\begin{lemma}\label{mainlemma}
	Let $\Omega$ and $L_1$ be as in \eqref{simple}, $U$ an open set.
	Suppose that $E\subset\Omega$ is $(U,h)$-sliding-almost-minimal. 
	For each $\tau>0$, we can find $\varepsilon(\tau)>0$ such that if $x\in
	E\cap L_1$ and $r>0$ are such that 
	\begin{equation}\label{densityconditions}
		B(x,r)\subset U, h(2r)\leq \varepsilon(\tau),
		\int_{0}^{2r}\frac{h(t)}{t}dt\leq\varepsilon(\tau),
		\theta(x,r)\leq\theta(x)+\varepsilon(\tau),
	\end{equation}
	then for every $\rho\in (0, 9r/10]$ there is a sliding minimal cone
	$Z_x^{\rho}$ centered at $x$, such that 
	\begin{equation}\label{eq:main1}
			d_{x,\rho}(E,Z_x^{\rho})\leq\tau
		\end{equation}
	and for any ball $ B(y,t)\subset B(x,\rho)$,
		\begin{equation}\label{eq:main}
			\left\vert \H^2(E\cap B(y,t))-\H^2(Z_x^{\rho}\cap B(y,t))
			\right\vert\leq\tau\rho^{2}\\
	\end{equation}
	Moreover, if $L_1\subset E$, then we can suppose that $L_1\subset Z_x^\rho$. 
\end{lemma}
This lemma is directly following from Proposition 30.19 in \cite{David:2014}.
If $L_1\subset E$, by the original proof in \cite[Proposition 30.19]
{David:2014}, we can go further and assert that $L_1\subset Z^{\rho}$, the
proof will be same, we do not even need to do any extra effort.

If $E$ is a sliding almost minimal set, then for any $x\in E\cap L_1$, any
blow-up limit of $E$ at $x$ is a sliding minimal cone, see 
\cite[Theorem 24.13]{David:2014}. Moreover, the density of any blow-up limit 
at origin is aways the value $\theta_E(x)$, see Proposition 7.31 \cite{David:2009}
and Corollary 29.53 in \cite{David:2014}. By Remark \ref{re:listofminimalcones},
\[
\theta_E(x)\in \left\{ \frac{1}{2},\frac{3}{4}, 1,
d_{T_+},\frac{3}{2},\frac{7}{4} \right\},
\]
where we denote by $d_{T_+}$ the density of cones of type $\mathbb{T}_{+}$ at
origin. In fact, $d_{T_+}=3\arccos(-1/3)/\pi-3/4\approx 1.07452$.

If $\theta_E(x)=\frac{1}{2}$, then $\theta_{Z_x^{\rho}}(x)=\frac{1}{2}$, thus
$Z_x^{\rho}$ is a sliding minimal cone of type $\mathbb{P}_{+}$ in $\Omega$ 
with sliding boundary $L_1$.

Similarly, if $\theta_E(x)=\frac{3}{4}$, then 
$\theta_{Z_x^{\rho}}(x)=\frac{3}{4}$, thus $Z_x^{\rho}$ is a sliding minimal 
cone of type $\mathbb{Y}_{+}$ in $\Omega$ with sliding boundary $L_1$.

If $L_1\subset E$, then the blow-up limit is sliding minimal
cone containing $L_1$. We know, by Theorem \ref{thm:MCCB},
that there are only three kinds of minimal cone which contain $L_1$.
That is, $L_1$ or $Z\cup L_1$, here $Z$ is a minimal
cone of type $\mathbb{P}_{+}$ or $\mathbb{Y}_{+}$. Thus, in the
case $L_1\subset E$, there are three possible values for $\theta_E(x)$, that 
is $1$, $\frac{3}{2}$ and $\frac{7}{4}$. In particular, if
$L_1\subset E$ and $\theta_E(x)=1$, then $Z_{x}^{\rho}=L_1$.
\begin{lemma}\label{le:D1}
	Let $E\subset\Omega$ be a sliding almost minimal set, $L_1\subset E$.
	If a blow-up limit of $E$ at $x\in L_1$ is the plane $L_1$, then there
	exists $r>0$ such that $E\cap B(x,r)=L_1\cap B(x,r)$. 
\end{lemma}
\begin{proof}
	Without loss of generality, we assume $x=0$. $L_1$ is a blow-up limit of
	$E$ at $0$. By corollary 29.53 in \cite{David:2014}, we get that 
	$\theta_E(0)=1$. Let $\tau>0$ be a small enough number, let 
	$\varepsilon(\tau)$ be as in Lemma \ref{mainlemma}. We take
	$0<\tau_2\leq \frac{\varepsilon(\tau)}{2}$, and let $\varepsilon(\tau_2)$ be
	as in Lemma \ref{mainlemma}. We take $r>0$ such that  
	\[
	(1+\varepsilon(\tau_2))\exp\left(
	\lambda\int_0^{\frac{r}{3}}\frac{h(2t)}{t}dt \right)<\frac{3}{2}
	\]
	and
	\[
	\theta_E(0,r)\leq 1+\varepsilon(\tau_2),
	\]
	where $\lambda$ is taken as in Proposition 5.24 in \cite{David:2009}.

	By lemma \ref{mainlemma}, for any $0<\rho\leq \frac{9r}{10}$,
	\begin{equation}\label{eq:App1}
		d_{0,\rho}(E,L_1)\leq \tau_2,
	\end{equation}
	and for all ball $B(y,t)\subset B(0,\rho)$,
	\begin{equation}\label{eq:App2}
		\left\vert \H^2(E\cap B(y,t))-\H^2(L_1\cap B(y,t)) \right\vert\leq \tau_2
		\rho^2.
	\end{equation}
	Thus, in particular, for any $y\in B(0,\frac{r}{3})\cap E$, 
	\[
	\theta_E\left( y,\frac{r}{3} \right)=\left(\pi\left( \frac{r}{3} \right)^2
	\right)^{-1}\H^2\left(E\cap B\left( y,\frac{r}{3} \right)\right)
	\leq 1+\frac{4\tau_2}{\pi}.
	\]
	For any $y\in B(0,\frac{r}{3})\cap L_1$, by Theorem 28.7 in \cite{David:2014},
	we get that
	\[
	1\leq \theta_E(y)\leq \theta_{E}\left( y,\frac{r}{3} \right)\exp\left(
	\lambda\int_0^{\frac{r}{3}}\frac{h(2t)}{t}dt \right)<\frac{3}{2},
	\]
	thus 
	\[
	\theta_E(y)=1.
	\]
	Therefore, for $y\in B\left( 0,\frac{r}{3} \right)\cap L_1$,
	\[
	\theta_E\left( y,\frac{r}{3} \right)\leq
	1+\frac{4\tau_2}{\pi}\leq \theta_E(y)+\varepsilon(\tau).
	\]
	By lemma \ref{mainlemma}, for any $0<\rho\leq\frac{3r}{10}$, 
	\begin{equation}\label{eq:App3}
		d_{y,\rho}(E,L_1)\leq \tau.
	\end{equation}

	We shall deduce, from equation \eqref{eq:App3}, that for any
	$0<\rho<\frac{r}{3}$, 
	\[
	E\cap B(0,\rho)=L_1\cap B(0,\rho).
	\]
	Once we have proved this, the desire result follows. We assume, for the sake
	of a contradiction, that 
	\[
	E\cap B(0,\rho)\setminus L_1\neq \emptyset.
	\]
	Let $z\in E\cap B(0,\rho)\setminus L_1$, and let $y$ be the projection of
	$z$ on $L_1$, then $0<\left\vert z-y \right\vert<\rho$. We choose $\rho'$
	such that 
	\[
	\left\vert z-y \right\vert<\rho'<\min\left\{\rho,\frac{\left\vert z-y
	\right\vert}{\tau} \right\}.
	\]
	From equation \eqref{eq:App3}, we can get that 
	\[
	\left\vert z-y \right\vert\leq \rho' d_{y,\rho'}(E,L_1)\leq
	\tau\rho'<\left\vert z-y \right\vert,
	\]
	absurd.
\end{proof}
\begin{lemma}\label{le:removeboundary}
	Let $\Omega$, $L_1$ and $U$ be as in Lemma \ref{mainlemma}, let $E\subset \Omega$
	be a $(U,h)$-sliding-almost-minimal set with $L_1\subset E$. Let
	$F=\overline{E\setminus L_1}$. Then $\H^2(F\cap L_1)=0$, and $F$ is also
	$(U,h)$-sliding-almost-minimal.
\end{lemma}
\begin{proof}
	We put $G=F\cap L_1$. Let $0<\varepsilon<1/100$. We assume, for the sake of 
	contradiction, that $\H^2(G)>0$. Since $G$ is a subset of $L_1$, it is 
	rectifiable, thus for $\H^2$-a.e. $x\in G$, $\theta_{G}(x)=1$. Without loss 
	of generality, we suppose that $\theta_{G}(0)=1$, then there exists a radius
	$r_1>0$ such that for all $0<r\leq r_1$,
	\begin{equation}\label{eq:RMB-1}
		\theta_{G}(0,r)\geq 1-\varepsilon.
	\end{equation}
	Since $E$ is sliding almost minimal, by Theorem 28.7 (almost monotonicity of
	density property) in \cite{David:2014}, we can find a radius $r_2>0$ such that 
	for all $0<r\leq r_2$, 
	\begin{equation}\label{eq:RMB0}
		1-\varepsilon \leq \frac{\theta_E(0,r)}{\theta_{E}(0)}\leq 1+\varepsilon.
	\end{equation}
	Since $E$ is sliding almost minimal and $L_1\subset E$, by Lemma
	\ref{mainlemma}, there exists $r>0$ such that for any $0<\rho\leq r$, there
	exists a sliding minimal cone $Z_{0}^{\rho}\supset L_1$ such that 
	\begin{equation}\label{eq:RMB1}
		d_{0,\rho}(E,Z_0^{\rho})\leq \varepsilon
	\end{equation}
	and for any ball $B(y,t)\subset B(0,\rho)$,
	\begin{equation}\label{eq:RMB2}
		\left\vert \H^2(E\cap B(y,t))-\H^2(Z_0^{\rho}\cap B(y,t)) \right\vert\leq
		\varepsilon \rho^2
	\end{equation}

	We take $0<\rho\leq \min\{ r,r_1,r_2 \}$, and consider a collection of balls 
	\[
	\mathcal{V}=
	\left\{ B(x,s)\midd {x\in G\cap B(0,\rho),s\leq \varepsilon \rho,
	B(x,s)\subset B(0,\rho) \atop \theta_{G}(x,s)\geq 1-\varepsilon,
	\theta_{E}(x,s)\geq (1-\varepsilon)\theta_{E}(x)}  \right\},
	\]
	it is a Vitali covering for $G\cap B(0,\rho)$. By a Vitali's covering 
	theorem for the Hausdorff measure, see for example, there exists a finite
	or countably infinite disjoint subcollection $\{ B_i \}_{i\in I}\subset
	\mathcal{V}$ such that 
	\begin{equation}\label{eq:RMB3}
		\H^2\left( G\cap B(0,\rho)\setminus \bigcup_{i\in I}B_i \right)=0.
	\end{equation}
	We now consider two balls $B_1'=B(y_1,t_1)$ and $B_1'=B(y_2,t_2)$, where 
	$y_1=(0,0,\frac{1+\varepsilon}{2}\rho)$,
	$y_2=(0,0,-\frac{1+\varepsilon}{2}\rho)$ and
	$t_1=t_2=\frac{1-\varepsilon}{2}\rho$. We can see that $B_1'\subset
	B(0,\rho)$ and $B_2'\subset B(0,\rho)$, thus by equation \eqref{eq:RMB2}, we
	can get that 
	\begin{equation}\label{eq:RMB4}
		\H^2(E\cap B_1')\geq \H^2(Z_0^{\rho}\cap B_1')-\varepsilon\rho^2
	\end{equation}
	and 
	\begin{equation}\label{eq:RMB5}
		\H^2(E\cap B_2')\geq \H^2(Z_0^{\rho}\cap B_2')-\varepsilon\rho^2.
	\end{equation}
	It is very easy to see that $\{ B_1',B_2' \}\cup \{ B_i \}_{i\in I}$ is a 
	family of disjoint balls and 
	\begin{equation}\label{eq:RMB6}
		B_1'\cup B_2'\cup \bigcup_{i\in I}B_i\subset B(0,\rho).
	\end{equation}
	For $i\in I$, we denote $B_i=B(x_i,s_i)$, then $x_i\in G$ and
	$\theta_E(x_i)\geq \frac{3}{2}$; otherwise, $\theta_E(x_i)=1$, any blow-up
	limit of $E$ at $x_i$ must be $L_1$, and by Lemma \ref{le:D1}, there is a
	small ball $B(x_i,r')$ such that $E\cap B(x_i,r')=L_1\cap B(x_i,r')$, that
	is impossible.

	By our choice of $\mathcal{V}$, we have that 
	\[
	\theta_E(x_i,s_i)\geq (1-\varepsilon)\theta_E(x_i)\geq
	\frac{3}{2}(1-\varepsilon),
	\]
	thus
	\begin{equation}\label{eq:RMB7}
		\H^2(E\cap B_i)\geq \frac{3}{2}(1-\varepsilon)\pi s_i^2\geq
		\frac{3}{2}(1-\varepsilon)\H^2(G\cap B_i),
	\end{equation}
	and combine with equations \eqref{eq:RMB3} and \eqref{eq:RMB-1}, to obtain
	\begin{equation}\label{eq:RMB8}
		\begin{aligned}
			\sum_{i\in I}\H^2(E\cap B_i)&\geq \frac{3}{2}(1-\varepsilon)\sum_{i\in
			I}\H^2(G\cap B_i)\\
			&\geq \frac{3}{2}(1-\varepsilon)\H^2(G\cap B(0,\rho))\\
			&\geq \frac{3\pi}{2}(1-\varepsilon)^2\rho^2.
		\end{aligned}
	\end{equation}

	Since $0\in G$, we have that $\theta_E(0)\geq \frac{3}{2}$, thus
	$\theta_E(0)=\frac{3}{2}$ or $\theta_E(0)=\frac{7}{4}$. 

	If $\theta_E(0)=\frac{3}{2}$, the sliding minimal $Z_{0}^{\rho}$ which we 
	chose in \eqref{eq:RMB1} can be written $Z_0^{\rho}=L_1\cap Z^{\rho}$, where
	$Z^{\rho}$ is a sliding minimal cone of type $\mathbb{P}_{+}$. In
	this case, $Z_0^{\rho}\cap B_i'$, $i=1,2$, are two disks with radius
	$\frac{1-\varepsilon}{2}\rho$, thus 
	\[
	\H^2(Z_0^{\rho}\cap B_i')=\pi\left(\frac{1-\varepsilon}{2}\rho\right)^2,
	\]
	combine this equation with equations \eqref{eq:RMB4}, \eqref{eq:RMB5},
	\eqref{eq:RMB6} and \eqref{eq:RMB8}, we can get that
	\begin{equation}\label{eq:RMB9}
		\begin{aligned}
			\H^2(E\cap B(0,\rho))&\geq \sum_{i=1}^{2}\H^2(E\cap B_i')+\sum_{i\in
			I}\H^2(E\cap B_i)\\
			&\geq 2\pi \left( \frac{1-\varepsilon}{2}\rho
			\right)^2+\frac{3}{2}\pi(1-\varepsilon)^2\rho^2-2\varepsilon\rho^2\\
			&> (2-5\varepsilon)\pi\rho^2,
		\end{aligned}
	\end{equation}
	but from equation \eqref{eq:RMB0}, we can get that 
	\begin{equation}\label{eq:RMB10}
		\H^2(E\cap B(0,\rho))\leq \frac{3}{2}(1+\varepsilon)\pi \rho^2,
	\end{equation}
	which contradict with equation \eqref{eq:RMB9}, because
	$0<\varepsilon<\frac{1}{100}$.

	If $\theta_E(0)=\frac{7}{4}$, a very similar calculation as above case, we
	can get that
	\[
	\H^2(Z_0^{\rho})=3\times \frac{\pi}{2}\left( \frac{1-\varepsilon}{2}\rho
	\right)^2,
	\]
	and 
	\[
	\H^2(E\cap B(0,\rho))\geq
	\frac{3}{4}(1-\varepsilon)^{2}\pi\rho^2+\frac{3}{2}(1-\varepsilon)^2\pi
	\rho^2 -2\varepsilon\rho^2>\left( \frac{9}{4}-\frac{11}{2}\varepsilon
	\right)\pi\rho^2,
	\]
	but from equation \eqref{eq:RMB0}, we obtain that 
	\[
	\H^2(E\cap B(0,\rho))\leq \frac{7}{4}(1+\varepsilon)\pi \rho^2,
	\]
	we also get a contradiction. We proved that $\H^2(F\cap L_1)=0$. We will go
	to prove that $F$ is also sliding almost minimal.

	Let $\{ \varphi_t \}_{0\leq t\leq 1}$ be any $\delta$-sliding-deformation. 
	Since $E$ is $(U,h)$-sliding-almost-minimal, applying Proposition 20.9 in
	\cite{David:2014}, we get that 
	\begin{equation}\label{eq:RMB11}
		\H^2(E\setminus \varphi_1(E))\leq \H^2(\varphi_1(E)\setminus
		E)+h(\delta)\delta^2.
	\end{equation}
	Since $\varphi_1(E)\supset L_1$, we can get that 
	\begin{equation}\label{eq:RMB12}
		\H^2(E\setminus \varphi_1(E))=\H^2( (E\setminus L_1)\setminus
		\varphi_1(E))=\H^2\left( (E\setminus L_1)\setminus \varphi(E\setminus L_1)
		\right).
	\end{equation}
	We know that $F=\overline{E\setminus L_1}$ and $\H^2(F\cap L_1)=0$, thus 
	\begin{equation}\label{eq:RMB13}
		\H^2(F\setminus \varphi_1(F))=\H^2(E\setminus \varphi_1(E)).
	\end{equation}
	Similarly, we can get that 
	\begin{equation}\label{eq:RMB14}
		\H^2(\varphi_1(E)\setminus E)=\H^2(\varphi(E\setminus L_1)\setminus E)\leq
		\H^2(\varphi(E\setminus L_1)\setminus (E\setminus L_1)),
	\end{equation}
	thus
	\begin{equation}\label{eq:RMB15}
		\H^2(\varphi_1(E)\setminus E)\leq \H^2(\varphi_1(F)\setminus F).
	\end{equation}
	From inequalities \eqref{eq:RMB11}, \eqref{eq:RMB13} and \eqref{eq:RMB15},
	we obtain that 
	\[
	\H^2(F\setminus \varphi_1(F))\leq \H^2(\varphi_1(F)\setminus
	F)+h(\delta)\delta^2.
	\]
	Applying Proposition 20.9 in \cite{David:2014}, we get that $F$ is
	$(U,h)$-sliding-almost-minimal.
\end{proof}
	If $\Omega$, $L_1$, $U$, $E$ and $F$ are as in lemma
	\ref{le:removeboundary}, and we suppose that $0\in F$, then $\theta_{F}(0)$ 
	can only take two values $\frac{1}{2}$ and $\frac{3}{4}$. 
	Indeed, since $L_1\subset E$, any blow-up limit $Z$ of $E$ at $0$ is a sliding
	minimal cone which contains the boundary $L_1$, thus $Z=L_1$ or $Z=L_1\cup
	Z'$, $Z'$ is a sliding minimal cone of type $\mathbb{P}_{+}$ or
	$\mathbb{Y}_{+}$, hence the density $\theta_{E}(0)$ can 
	only take three values, $1$, $\frac{3}{2}$ and $\frac{7}{4}$. But if
	$\theta_E(0)=1$, then by Lemma \ref{le:D1}, we can see that $0\not\in F$.
	Therefore, $\theta_E(0)=\frac{3}{2}$ or $\frac{7}{4}$. We see that 
	\begin{equation}\label{eq:MWOB3}
		\theta_E(0,r)=\theta_F(0,r)+1,
	\end{equation}
	thus
	\[
	\theta_E(x)=\theta_F(x)+1,
	\]
	and $\theta_F(0)=\frac{1}{2}$ or $\frac{3}{4}$.

\begin{lemma}\label{le:BDAP}
	Let $\Omega$ and $L_1$ be as in \eqref{simple}, $\Pi_{L_1}:\mathbb{R}^{3}\to
	L_1$ be the orthogonal projection onto $L_1$. Suppose that $U$ is an open
	set, $E\subset\Omega$ is $(U,h)$-sliding-almost-minimal, $0\in E\cap L_1\cap
	U$ and $\theta_E(0)=\frac{1}{2}$ or $\frac{3}{4}$. If $\varepsilon>0$ is small
	enough, $B(0,2r)\subset U$, and for any $x\in F\cap L_1\cap B(0,r)$, and any 
	$0<\rho\leq 3r/5$, there exists a sliding minimal cone $Z_x^{\rho}$ of type 
	$\mathbb{P}_{+}$ or $\mathbb{Y}_{+}$ centered at $x$, such that 
	\[
	d_{x,\rho}(E,Z_{x}^{\rho})\leq \varepsilon,
	\]
	then for any $z\in E\cap B(0,r/5)$, we can find a point $a\in E\cap L_1\cap 
	B(0,3r/5)$ such that 
	\begin{equation}\label{eq:BDAP1}
		\left\vert \Pi_{L_1}(z)-a \right\vert\leq 8\varepsilon \left\vert z-a \right\vert.
	\end{equation}
\end{lemma}
\begin{proof}
	For any $z\in E\cap B(0,r/5)$, we put $z'=\Pi_{L_1}(z)$. We take a point
	$a\in E\cap L_1$ such that 
	\[
	|z'-a|\leq (1+\varepsilon)\dist(z',E\cap L_1).
	\]
	If $z'\in E\cap L_1$, $a=z'\in B(0,r/5)$, then nothing needs to be done. 
	If $z'\not\in E\cap L_1$, we claim that $a$ is a point which we desire.

	It is quite easy to see that $a\in B(0,3r/5)$; otherwise 
	\[
	\frac{2r}{5}\leq |z'-a|\leq (1+\varepsilon)\dist(z',E\cap L_1)\leq
	(1+\varepsilon)|z'-0|\leq (1+\varepsilon)\frac{r}{5};
	\]
	this gives a contradiction.

	We put $\rho=2|a-z|$. Since $d_{a,\rho}(E,Z_a^{\rho})\leq \varepsilon$, and
	$Z_a^{\rho}$ is perpendicular to $L_1$, we can find $z''\in Z_a^{\rho}\cap
	L_1$ such that $|z'-z''|\leq \varepsilon\rho$. 
	
	We claim that $|z''-a|\leq 3\varepsilon\rho$; once we have proved our claim,
	we can get that 
	\[
	|\Pi_{L_1}(z)-a|\leq |z'-z''|+|z''-a|\leq 4\varepsilon\rho=8\varepsilon |x-z|
	\]
	
	We assume, for the sake of a contradiction, that $|z''-a|>3\varepsilon\rho$,
	then 
	\[
	|a-z'|\geq |a-z''|-|z'-z''|>2\varepsilon\rho.
	\] 
	If $B(z'',3\varepsilon\rho/2)\cap E\cap L_1\neq\emptyset$, we take 
	$x\in B(z'',3\varepsilon\rho/2)\cap E\cap L_1$, then 
	\[
	|z'-x|\leq |z'-z''|+|z''-x|\leq \frac{5}{2}\varepsilon\rho,
	\]
	and 
	\[
	|z-x'|\geq \dist(z',E\cap L_1)\geq \frac{1}{1+\varepsilon}|z'-a|\geq
	\frac{2\varepsilon\rho}{1+\varepsilon},
	\]
	thus
	\[
	\frac{2\varepsilon\rho}{1+\varepsilon}\leq \frac{5}{2}\varepsilon\rho;
	\]
	this is a contradiction.

	If $B(z'',3\varepsilon\rho/2)\cap E\cap L_1=\emptyset$, we can construct a  
	projection to show that $E$ is not almost minimal.

\end{proof}
\begin{lemma}\label{lemmaplane}
	Let $\Omega$, $L_1$ be as in \eqref{simple}, U an open set with $0\in U$.
	Suppose that $E\subset \Omega$ is $(U,h)$-sliding-almost-minimal.
	If $\theta_E(0)=\frac{1}{2}$, then for each small $\tau>0$, 
	we can find a radius $r>0$, a sliding minimal cone $Z$ of type 
	$\mathbb{P}_{+}$ and a biH\"older map $\phi: B(0,3r/2)\cap \Omega\to
	B(0,2r)\cap \Omega$ such that 
	\begin{equation}\label{eq:lemmaplane}
		\begin{gathered}
			\phi(x)\in L_1\text{ for }x\in L_1,\| f-\id \|_{\infty}\leq \tau,\\
			(1+\tau)^{-1}\left\vert z-y \right\vert^{1+\tau}\leq \left\vert
			\phi(z)-\phi(y) \right\vert\leq (1+\tau)\left\vert z-y
			\right\vert^{\frac{1}{1+\tau}},\\
			B(0,r)\cap\Omega\subset \phi\left( B\left( 0,\frac{3r}{2} \right)\cap
			\Omega \right)\subset B(0,2r)\cap \Omega,\\
			E\cap B(0,r)\subset \phi\left(Z\cap B\left( 0,\frac{3r}{2}
			\right)\right)\subset E\cap B(0,2r).
		\end{gathered}
	\end{equation}
\end{lemma}
\begin{proof}
	We can assume that $U$ is an open ball $B(0,R)$ for some $R>0$. 
	For any $\tau\in (0,1]$, we let $\varepsilon(\tau)$ be as in Lemma
	\ref{mainlemma}, we suppose that $\tau$ is so small, that 
	\[
	(1+\varepsilon(\tau))e^{(\lambda+\alpha) \varepsilon(\tau)}<\frac{3}{2},
	\]
	where $\lambda$ is taken as in Proposition 5.24 in \cite{David:2009}, and
	$\alpha$ is taken as in Theorem 28.7 in \cite{David:2014}.
	Let $\tau_2>0$ and $\varepsilon(\tau_2)$ be as in Lemma \ref{mainlemma} and
	such that $100\tau_2\leq \tau $ and
	\[
	\left( \frac{1}{2}+\varepsilon(\tau_2)
	\right)e^{\alpha \varepsilon(\tau)}<\frac{3}{4}.
	\]
	We take $0<\tau_{1}\leq\min\{\tau_2,\varepsilon(\tau_2)\}/100$, and let
	$\tau_{1}$, $\varepsilon(\tau_{1})$ be also as in Lemma \ref{mainlemma}. 
	We always suppose that 
	$\varepsilon(\tau_1)<\varepsilon(\tau_2)<\varepsilon(\tau)$.

	By Theorem 28.7 in \cite{David:2014}, we can find $r_{0}\in (0,R)$ such that 
	\[
	h(2r_0)\leq \varepsilon(\tau_1),\ A(r_0)\leq \varepsilon(\tau_1),\
	\theta_E(0,r_{0})\leq \frac{1}{2}+\varepsilon(\tau_{1}),
	\]
	where 
	\[
	A(r)=\int_{0}^{2r}\frac{h(t)}{t}dt.
	\]
	By using Lemma \ref{mainlemma}, for any $r\in (0,9r_{0}/10]$, there exists a
	minimal cone $Z^r$ of type $\mathbb{P}_{+}$ center at $0$ such that 
	\begin{equation}\label{measurecontrolplane0}
			d_{0,r}(E,Z^r)\leq \tau_{1}
	\end{equation}
	and for any ball $B(y,t)\subset B(0,r)$,
	\begin{equation}\label{measurecontrolplane}
			\left\vert\H^{2}(Z^r\cap B(y,t))-\H^{2}(E\cap B(y,t)) 
			\right\vert\leq \tau_{1} r^{2}.
	\end{equation}

	First, we consider any point $x\in E\cap L_{1}\cap B(0, r_{0}/2)$. 
	By \eqref{measurecontrolplane}, if we take $r=9r_0/10$, we will get that 
	\begin{align*}
		\H^{2}(E\cap B(x,t))&\leq \H^2\left( Z^{9r_0/10}\cap
		B(x,t)\right)+\tau_1 \left( \frac{9r_0}{10} \right)^2\\
		&\leq \frac{\pi}{2}t^2+\tau_{1}\left( \frac{9r_0}{10} \right)^2
	\end{align*}
	from this inequality, by taking $t=r_{0}/3$, we can get that
	\begin{equation}\label{eq:LP1}
		\theta_E\left(x,\frac{r_{0}}{3}\right)< \frac{1}{2}+30\tau_1 <
		\frac{1}{2}+\varepsilon(\tau_2).
	\end{equation}
	By using Theorem 28.7 in \cite{David:2014}, we get that
	\[
	\theta_E(x)\leq \theta_E\left( x,\frac{r_0}{3} \right)e^{\alpha
	A(r_0/3)}<\frac{3}{4},
	\]
	thus
	\begin{equation}\label{eq:LP2}
		\theta_E(x)=\frac{1}{2}.
	\end{equation}

	By Lemma \ref{mainlemma}, we can find minimal sliding cone
	$Z_x^{\rho}$ for any $0<\rho\leq 3r_0/10$ such that 
	\begin{equation}\label{eq:LP3}
			d_{x,\rho}(E,Z_x^{\rho})\leq \tau_2
	\end{equation}
	 and for any ball $B(y,t)\subset B(x,\rho)$,
	\begin{equation}\label{eq:LP3.5}
			\left\vert \H^2(E\cap B(y,t))-\H^2(Z_x^{\rho}\cap B(y,t)) \right\vert\leq
			\tau_2\rho^2.
	\end{equation}

	We now consider any point $z\in E\cap B(0,r_0/10)\setminus L_1$.
	From Lemma \ref{le:BDAP}, we get that 
	\[
	\left\vert\Pi_{L_1}(z)-x\right\vert\leq 8\tau_2\left\vert z-x \right\vert,
	\]
	thus
	\begin{equation}\label{eq:LP5}
		\dist(z,L_1)=\left\vert z-\Pi_{L_1}(z) \right\vert\geq \left\vert z-x
		\right\vert-\left\vert \Pi_{L_1}(z)-x \right\vert\geq
		(1-8\tau_2)\left\vert z-x \right\vert.
	\end{equation}

	We take $r_1=\frac{1}{2}\dist(z,L_1)$,
	and $\rho=\left\vert z-x \right\vert+r_1$, then $\rho<\frac{3r_0}{10}$. We take
	$Z_x^{\rho}$ as in \eqref{eq:LP3}, then $B(z,r_1)\subset B(x,\rho)$, thus
	\[
	\H^2(E\cap B(z,r_1))\leq \H^2(Z_{x}^{\rho}\cap B(z,r))+\tau_2\rho^2\leq \pi
	r_1^2+\tau_2\rho^2,
	\]
	hence 
	\begin{equation}\label{eq:LP4}
		\theta_E(z,r_1)=\frac{\H^{2}(E\cap B(z,r_1))}{\pi r_1^2} \leq
		1+\frac{\tau_2\rho^2}{\pi r_1^2}< 1+10\tau_2<1+\varepsilon(\tau),
	\end{equation}
	but we know that 
	\[
	\theta_E(z)\geq 1,
	\]
	hence by using a monotonicity property, Proposition 5.24 in \cite{David:2009},
	we have that
	\[
	\theta_E(z)\leq\theta_E(z, r_1)\exp(\lambda A( r_1 ))<\frac{3}{2},
	\]
	thus
	\begin{equation}\label{eq:LP5}
		\theta_E(z)=1.
	\end{equation}

	For any $r\in (0, \frac{9}{10}r_1]$,
	we can apply Lemma 16.11 in \cite{David:2009}, 
	there exists a plane $Z(z,r)$ through $z$ such that
	\begin{equation}\label{eq:LP6}
		d_{z,r}(E,Z(z,r))\leq \tau.
	\end{equation}

	For any	$r\in(\frac{9}{10}r_1, \frac{1}{5}r_0]$,
	we put $\rho_r=\left\vert z-x\right\vert+r$, then $\rho_r\leq
	\frac{3}{10}r_0$, and $B(z,r)\subset B(x,\rho_r)$. We take $Z_x^{\rho_r}$ as
	in \eqref{eq:LP3}, then 
	\begin{equation}\label{eq:LP7}
		d_{z,r}(E,Z_x^{\rho_r}) \leq \frac{\rho_r}{r}d_{x,\rho_r}(E,Z(x,\rho_r))
		\leq \frac{\rho_r}{r}\tau_2\leq 5\tau_2.
	\end{equation}
	We do not know whether or not the sliding minimal cone $Z_x^{\rho_r}$ passes
	through the point $z$, but we can do a translation of $Z_x^{\rho_r}$ such that
	it is centered at $\Pi_{L_1}(z)$, we denote it by $Z(z,r)$, i.e.
	$Z(z,r)=Z_x^{\rho_r}-(x-\Pi_{L_1}(z))$. Then $Z(z,r)$ is a sliding minimal
	cone contains $z$, and 
	\begin{equation}\label{eq:LP8}
		d_{z,r}(E,Z(z,r))\leq \frac{\left\vert x-\Pi_{L_1}(z)
		\right\vert}{r}+d_{z,r}(E,Z_x^{\rho_r})<20\tau_2<\tau.
	\end{equation}
	It follows from \eqref{eq:LP6} and \eqref{eq:LP8} that, for any $z\in E\cap
	B(0,r_0/10)\setminus L_1$, for any $r\in (0,r_0/5]$, there exist a cone
	$Z(z,r)$ such that 
	\begin{equation}\label{eq:LP9}
		d_{z,r}(E,Z(z,r))\leq \tau,
	\end{equation}
	where $Z(z,r)$ is a plane when $r$ is small, $Z(x,r)$ is a half plane when $r$
	is large.

	From the inequalities \eqref{eq:LP3.5} and \eqref{eq:LP9}, we get that, for any
	$x\in E\cap B(0,r_0/10)$ and any $r\in (0,r_0/5]$, we can find a cone $Z(x,r)$
	though $x$ such that 
	\[
	d_{x,r}(E,Z(x,r))\leq \tau.
	\]
	where $Z(x,r)$ is a minimal cone when $0<r<\dist(x,L_1)$, and $Z(x,r)$ is a 
	sliding minimal cone of type $\mathbb{P}_{+}$ when $\dist(x,L_1)\leq r\leq
	r_0/5$.

	By Corollary \ref{co:PT}, we can find a biH\"older map 
	$\phi: B(0,3r_0/20)\cap \Omega\to B(0,r_0/5)\cap \Omega$ and a sliding 
	minimal cone $Z_0$ of type $\mathbb{P}_{+}$ such that \eqref{eq:lemmaplane}
	holds with $r=r_0/10$.	
\end{proof}
\begin{lemma}\label{le:PUFP2}
	Let $\Omega$, $L_1$ be as in \eqref{simple}, U an open set with $0\in U$.
	Suppose that $E\subset \Omega$ is $(U, h)$-sliding-almost-minimal. 
	If $\theta_E(0)=\frac{3}{2}$, then for each small
	$\tau>0$, we can find a radius $r>0$, a biH\"older map
	$\phi: B(0,3r/2)\cap \Omega\to B(0,2r)\cap \Omega$ and a sliding minimal
	cone $Z$ of type $\mathbb{P}_{+}$ such that
	\begin{equation}\label{eq:PUFP2}
		\begin{gathered}
			\phi(x)\in L_1\text{ for }x\in L_1,\| f-\id \|_{\infty}\leq \tau,\\
			(1+\tau)^{-1}\left\vert z-y \right\vert^{1+\tau}\leq \left\vert
			\phi(z)-\phi(y) \right\vert\leq (1+\tau)\left\vert z-y
			\right\vert^{\frac{1}{1+\tau}},\\
			B(0,r)\cap\Omega\subset \phi\left( B\left( 0,\frac{3r}{2} \right)\cap
			\Omega \right)\subset B(0,2r)\cap \Omega,\\
			E\cap B(0,r)\subset \phi\left( (Z\cup L_1)\cap B\left( 0,\frac{3r}{2}
			\right)\right)\subset E\cap B(0,2r).
		\end{gathered}
	\end{equation}
\end{lemma}
\begin{proof}
	We put $F=\overline{E\setminus L_1}$, then $F$ is also
	$(U,\delta,h)$-sliding-almost-minimal. By lemma \ref{lemmaplane}, for each small
	$\tau>0$, we can find $r>0$, a biH\"older map $\phi: B(0,3r/2)\cap\Omega\to
	B(0,2r)\cap\Omega$ and a sliding minimal cone $Z$ of type $\mathbb{P}_{+}$ 
	such that
	\begin{equation}\label{eq:PUFP3}
		\begin{gathered}
			\phi(x)\in L_1\text{ for }x\in L_1\text{ and } \| f-\id \|_{\infty}\leq
			\tau,\text{ and}\\
			(1+\tau)^{-1}\left\vert x-y \right\vert^{1+\tau}\leq \left\vert
			\phi(x)-\phi(y) \right\vert\leq (1+\tau)\left\vert x-y
			\right\vert^{\frac{1}{1+\tau}},\\
			B(0,r)\cap\Omega\subset \phi\left( B\left( 0,\frac{3r}{2} \right)\cap
			\Omega \right)\subset B(0,2r)\cap \Omega,\\
			F\cap B(0,r)\subset \phi\left(Z\cap B\left( 0,\frac{3r}{2}
			\right)\right)\subset F\cap B(0,2r).
		\end{gathered}
	\end{equation}
	Thus 
	\[
	E\cap B(0,r)\subset\phi\left( (Z\cup L_1)\cap B\left( 0,\frac{3r}{2} \right)
	\right)\subset E\cap B(0,2r).
	\]
\end{proof}
\begin{remark}\label{re:plane}
	Suppose that $\Omega$, $L_1$ and $U$ are as in Lemma \ref{lemmaplane}, and
	that $E\subset \Omega$ is a $(U,h)$-sliding-almost-minimal set with
	$\theta_E(0)=\frac{1}{2}$ or $\frac{3}{2}$. If $\tau\in (0,1)$ is small
	enough, we can find $\varepsilon'(\tau)>0$ such that when the radius $r>0$ is 
	such that 
	\[
	B(0,10r)\subset U, h(20r)\leq \varepsilon'(\tau),
	\int_{0}^{20r}\frac{h(t)}{t}dt\leq\varepsilon'(\tau),\theta_E(0,10r)\leq
	\theta_E(0)+\varepsilon'(\tau)
	\]
	then for any $x\in E\cap B(0,r)$ and any $0<t\leq 2r$, we can find a cone or
	sliding minimal cone $Z(x,t)$ that depends on $t$ such that 
	\[
	d_{x,t}(E,Z(x,t))\leq \tau,
	\]
	where $Z(x,t)$ is a minimal cone when $0<t<\dist(x,L_1)$, and $Z(x,t)$ is a 
	sliding minimal cone when $\dist(x,L_1)\leq t\leq 2r$.
\end{remark}
Indeed, when we look at the proof of Lemma \ref{lemmaplane}, we let
$\tau\in (0,1)$ be such that 
\[
	\left( \frac{1}{2}+\varepsilon(\tau)
	\right)e^{(\lambda+\alpha)\varepsilon(\tau)}<\frac{3}{4}.
	\]
	Then we take 
	\[
	\tau_1=\min\left\{\frac{\tau}{10^4},
	\frac{1}{100}\varepsilon\left(\frac{\tau}{100}\right)\right\},
	\]
	and let $\varepsilon(\tau_1)$ be as in Lemma \ref{mainlemma}. Finally, 
$\varepsilon'(\tau)=\varepsilon(\tau_1)$ will be what we desire.

\begin{lemma}\label{lemmay}
	Let $\Omega$, $L_1$ be as in \eqref{simple}, $U$ an open set with $0\in U$.
	Suppose that $F\subset \Omega$ is an $(U,h)$-sliding-almost-minimal set.
	If $\theta_F(0)=\frac{3}{4}$, then for each small $\tau>0$, we can 
	find a radius $r>0$, a biH\"older map $\phi: B(0,3r/2)\cap \Omega\to
	B(0,2r)\cap\Omega$ and a sliding minimal cone of type $\mathbb{Y}_{+}$ such that
	\begin{equation}\label{eq:lemmay}
		\begin{gathered}
			\phi(x)\in L_1\text{ for }x\in L_1, \| \phi-\id \|_{\infty}\leq \tau,\\
			(1+\tau)^{-1}\left\vert z-y \right\vert^{1+\tau}\leq \left\vert
			\phi(z)-\phi(y) \right\vert\leq (1+\tau)\left\vert z-y
			\right\vert^{\frac{1}{1+\tau}},\\
			B(0,r)\cap\Omega\subset \phi\left( B\left( 0,\frac{3r}{2} \right)\cap
			\Omega \right)\subset B(0,2r)\cap \Omega,\\
			E\cap B(0,r)\subset \phi\left(Z\cap B\left( 0,\frac{3r}{2}
			\right)\right)\subset E\cap B(0,2r).
		\end{gathered}
	\end{equation}
\end{lemma}
\begin{proof}
	As in Lemma \ref{lemmaplane}, we can assume $U$ is an open ball $B(0,R)$ 
	for some $R>0$. Let $\tau>0$ be a positive number, and let
	$\varepsilon(\tau)$ be as in Lemma \ref{mainlemma}; we suppose $\tau$
	small enough so that
	\[
	(1+\varepsilon(\tau))e^{(\lambda+\alpha)\varepsilon(\tau)}<\frac{3}{2},
	\]
	where $\lambda$ is taken as in Proposition 5.24 in \cite{David:2009}, and
	$\alpha$ is taken as in Theorem 28.7 in \cite{David:2014}.
	Let $\tau_2>0$ and $\varepsilon(\tau_2)$ be as in Lemma \ref{mainlemma} 
	so that $100\tau_2\leq \tau $ and
	\[
	\left( \frac{1}{2}+\varepsilon(\tau_2)
	\right)e^{\alpha \varepsilon(\tau)}<\frac{3}{4},\ \left(
	\frac{3}{2}+\varepsilon(\tau_2) \right)e^{\lambda\varepsilon(\tau)}<d_T,
	\]
	where $d_T$ is the constant which is considered in Lemma 14.12 in
	\cite{David:2009}. We take $0<\tau_1\leq\min\{\tau_2,\varepsilon(\tau_2)\}/100$.
	Let $\tau_1$ and $\varepsilon(\tau_1)$ be as in Lemma \ref{mainlemma}. We
	suppose that $\varepsilon(\tau_1)<\varepsilon(\tau_2)<\varepsilon(\tau)$.

	As in the proof of Lemma \ref{mainlemma}, we put	$F=\overline{E\setminus L_1}$,
	then $F$ is also $(U,\delta,h)$-sliding-almost-minimal.

	By Theorem 28.7 in \cite{David:2014}, there exist $0<r_0<R$ such that 
	\[
	h(2r_0) \leq \varepsilon(\tau_1),\ A(r_0)<\varepsilon(\tau_1)
	\]
	and 
	\[
	\theta_F(0,r_{0})\leq \frac{3}{4}+\varepsilon(\tau_1),
	\]
	where 
	\[
	A(r)=\int_{0}^{2r}\frac{h(t)}{t}dt.
	\]
	By using Lemma \ref{mainlemma}, for any $\rho\in (0,9r_{0}/10]$ there exists a
	minimal cone $Z^{\rho}$ of type $\mathbb{Y}_{+}$ center at $0$ such that 
	\begin{equation}\label{eq:LY1}
		\begin{gathered}
			d_{0,\rho}(F,Z^{\rho})\leq \tau_{1},\\
			\left\vert\H^{2}(Z^{\rho}\cap B(y,t))-\H^{2}(E\cap B(y,t)) \right\vert\leq
			\tau_{1} \rho^{2},\\
			\text{ for any ball }B(y,t)\subset B(0,\rho).
		\end{gathered}
	\end{equation}

	First, for any $x\in F\cap L_{1}\cap B(0, r_{0}/2)\setminus \{ 0 \}$,
	we take $\rho=2\left\vert x \right\vert$ and $t=\left\vert x \right\vert$, 
	then by \eqref{eq:LY1}, we have 
	\begin{align*}
		\H^{2}(F\cap B(x,t))&\leq \H^2\left( Z^{\rho}\cap
		B(x,t)\right)+\tau_1 \rho^2\\
		&\leq \frac{\pi}{2}t^2+\tau_{1}\rho^2
	\end{align*}
	from this inequality, we can get that
	\begin{equation}\label{eq:LY3}
		\theta_F\left( x,\left\vert x \right\vert\right)=\frac{\H^{2}(E\cap 
		B( x,\left\vert x \right\vert))}{\pi\left\vert x \right\vert^2 }\leq
		\frac{1}{2}+4\tau_1<\frac{1}{2}+\varepsilon(\tau_2)
	\end{equation}
	Applying Theorem 28.7 in \cite{David:2014}, we get that
	\[
	\theta_F(x)\leq \theta_F(x,\left\vert x \right\vert)e^{\alpha
	A(\left\vert x \right\vert)}<\frac{3}{4},
	\]
	thus 
	\[
	\theta_F(x)=\frac{1}{2}.
	\]

	Taking $r_1=\left\vert x \right\vert$,
	we get from \eqref{eq:LY3} that
	\[
	\theta_{F}(x,r_1)\leq \theta_F(x)+\varepsilon(\tau_2).
	\]
	By Lemma \ref{mainlemma}, for any $0<\rho\leq 9r_1/10$, there exists
	a sliding minimal cone $Z_{x}^{\rho}$ centered at $x$ of type 
	$\mathbb{P}_{+}$ such that 
	\begin{equation}\label{eq:LY4}
		d_{x,r}(F,Z_x^{\rho})\leq \tau_2,
	\end{equation}
	and for any ball $B(y,t)\subset B(x,r)$,
	\begin{equation}\label{eq:LY5}
		\left\vert \H^2(F\cap B(y,t))-\H^2(Z_x^{\rho}\cap B(y,t)) \right\vert \leq
		\tau_2 \rho^2.
	\end{equation}
	For $9r_1/10\leq \rho \leq 3r_0/10$, we have that
	\[
	d_{x,\rho}(F,Z^{r_1+\rho})\leq \frac{r_1+\rho}{\rho}d_{0,r_1+\rho}(F,Z^{r_1+\rho})
	\leq \frac{19}{9}\tau_1.
	\]
	Since $d_{0,r_1+\rho}(F, Z^{r_1+\rho})\leq \tau_1$, there exists a point
	$x'\in Z^{r_1+\rho}\cap L_1$ such that $\left\vert x-x' \right\vert\leq
	(r_1+\rho)\tau_1$. We take $Z(x,\rho)=Z^{r_1+\rho}+x-x'$, that is a 
	translation of $Z^{r_1+\rho}$; it is a siliding minimal through the point $x$,
	and
	\begin{equation}\label{eq:LY6}
		d_{x,\rho}(F,Z(x,\rho))\leq \frac{\left\vert x-x' \right\vert}{\rho}+
		d_{x,\rho}(F,Z^{r_1+\rho})< 5\tau_1<\tau_2.
	\end{equation}

	It follows from \eqref{eq:LY4} and \eqref{eq:LY6} that, for any 
	$x\in F\cap L_1\cap B(0, r_0/2)$, and any $0<\rho<3r_0/10$, there exists a
	sliding minimal cone $Z(x,\rho)$ centered at $x$, either of type 
	$\mathbb{P}_{+}$ or of type $\mathbb{Y}_{+}$, such that 
	\begin{equation}\label{eq:LY7}
		d_{x,\rho}(F,Z(x,\rho))\leq \tau_2.
	\end{equation}

	Next, We consider $z\in (F\setminus L_1)\cap B(0,r_0/10)$. If 
	$\dist(z,L_1)<\frac{1}{3}\left\vert z \right\vert$,
	we take a point $a\in F\cap L_1\cap B(0,\frac{r_0}{5})$ such that 
	\[
	\left\vert\Pi_{L_1}(z)-a\right\vert\leq 8\tau_2\left\vert z-a \right\vert.
	\]
	We take $r_2=\frac{1}{2}\dist(x,L_1)$ and $\rho=\left\vert z-a
	\right\vert+r_2$, then
	\[
	\left\vert z-a \right\vert\leq \left\vert z-\Pi_{L_1}(z)
	\right\vert+\left\vert \Pi_{L_1}(z)-a \right\vert\leq
	2r_2+8\tau_2\left\vert z-a \right\vert,
	\]
	thus
	\[
	\left\vert z-a \right\vert\leq \frac{2}{1-8\tau_2}r_2,
	\]
	and 
	\[
	\rho\leq \left( 1+\frac{2}{1-8\tau_2} \right)r_2\leq \frac{75}{23}r_2.
	\]
	Since $2r_2=\dist(z,L_1)\leq \frac{1}{3}\left\vert z \right\vert$, we have
	that 
	\[
	\left\vert\Pi_{L_1}(z)\right\vert\geq 2\sqrt{2}\left\vert z-\Pi_{L_1}(z)
	\right\vert\geq 4\sqrt{2}r_2,
	\]
	thus
	\[
	\left\vert a \right\vert\geq \left\vert \Pi_{L_1}(z)
	\right\vert-\left\vert \Pi_{L_1}(z)-a \right\vert\geq
	4\sqrt{2}r_2-\frac{16\tau_2}{1-8\tau_2}r_2>4r_2.
	\]
	Hence 
	\[
	\rho\leq \frac{9}{10}\left\vert a \right\vert.
	\]
	Consider the sliding minimal cone $Z_a^{\rho}$ as in \eqref{eq:LY4}; it is a
	minimal cone centered at point $a$ of type $\mathbb{P}_{+}$. Since
	$B(z,r_2)\subset B(a,\rho)$, we deduce from \eqref{eq:LY5} that
	\begin{align*}
		\H^2 \left(F\cap B(z,r_2 )\right)&\leq \H^2
		\left(Z_{a}^{\rho} \cap B(z, r_2)\right)+\tau_2 \rho^2\\
		&\leq \pi r_2^2+\tau_2 \rho^2,
	\end{align*}
	thus 
	\begin{equation}\label{eq:LY8}
		\theta_F\left(z,r_2\right)\leq
		1+\left(\frac{73}{23}\right)^2\frac{\tau_2}{\pi}\leq 1+ 4\tau_2
		\leq 1+\varepsilon(\tau).
	\end{equation}
	Applying Proposition 5.24 in \cite{David:2009}, we can get that
	\[
	\theta_F(z)\leq \theta_F(z,r_2)e^{\lambda A(r_2)}<\frac{3}{2},
	\]
	thus
	\[
	\theta_F(z)= 1.
	\]
	By Lemma 16.11 in \cite{David:2009}, for any $\rho\in (0,9r_2/10]$, there
	exist a plane $Z(z,\rho)$ through $z$ such that 
	\begin{equation}\label{eq:LY9}
		d_{z,\rho}(F,Z(z,\rho))\leq \tau.
	\end{equation}
	For $\rho\in (9r_2/10,r_0/10]$, we put $r_{\rho}=\left\vert z-a
	\right\vert+\rho$, then 
	\[
	r_{\rho}\leq \frac{r_0}{5}+\frac{r_0}{10}\leq \frac{3r_0}{10}.
	\]
	Consider the sliding minimal cone $Z(a,r_{\rho})$ as in \eqref{eq:LY7}, 
	we can get that 
	\[
	d_{z,\rho}(F,Z(a,r_{\rho})\leq \frac{r_{\rho}}{\rho}d_{a,r}(F,Z(a,r_{\rho}))
	\leq \left( 1+\frac{\left\vert z-a \right\vert}{\rho} \right)\tau_2\leq
	\frac{7}{2}\tau_2.
	\]
	We now take 
	\[
	Z(z,\rho)=Z(a,r_{\rho})+\Pi_{L_1}(z)-a.
	\]
	It is a sliding minimal cone centered at $\Pi_{L_1}(z)$, thus through the point
	$z$, which satisfy that 
	\begin{equation}\label{eq:LY10}
		d_{z,\rho}(F,Z(z,\rho))\leq \frac{\left\vert \Pi_{L_1}(z)-a
		\right\vert}{\rho}+d_{z,\rho}(F,Z(a,r_{\rho}))<4\tau_2.
	\end{equation}
	It follows from \eqref{eq:LY9} and \eqref{eq:LY10} that, in the case 
	$z\in B(0,r_0/10)\cap F\setminus L_1$ and
	$\dist(z)<\left\vert z \right\vert/3$, for $0<\rho\leq \frac{r_0}{10}$, we can
	find a cone $Z(z,\rho)$ such that 
	\begin{equation}\label{eq:LY18}
		d_{z,\rho}(F,Z(z,\rho))\leq \tau,
	\end{equation}
	where $Z(z,\rho)$ is a minimal cone when $\rho$ is small, and $Z(z,\rho)$ is a
	sliding minimal cone when $\rho$ is large.

	We now consider the case when $z\in (F\setminus L_1)\cap B(0,r_0/10)$ with 
	$\dist(z,L_1)\geq \frac{1}{3}\left\vert z \right\vert$. We take
	$r_3=\dist(z,L_1)$, and put $\rho_3=\left\vert z \right\vert+r_3$, then
	$\rho_3\leq 4r_3<\frac{4r_0}{10}$. We take $Z^{\rho_3}$ a minimal
	cone as in \eqref{eq:LY1}, then we can get that 
	\[
	\H^2(F\cap B(z,r_3))\leq \H^2(Z^{\rho_3}\cap B(z,r_3))+\tau_1\rho_3^2\leq
	\frac{3}{2}\pi r_3^2+\tau_1\rho_3^2,
	\]
	thus 
	\begin{equation}\label{eq:LY11}
		\theta_F(z,r_3)\leq
		\frac{3}{2}+\frac{16}{\pi}\tau_1<\frac{3}{2}+\varepsilon(\tau_2).
	\end{equation}
	Applying Proposition 5.24 in \cite{David:2009}, we get that 
	\[
	\theta_F(z)\leq \theta_F(z,r_3)e^{\lambda A(r_3)}<d_T,
	\]
	thus $\theta_F(z)=1$ or $\theta_F(z)=\frac{3}{2}$.

	Case 1. If $\theta(z)=\frac{3}{2}$, then for any $0<\rho\leq
	\frac{9}{10}r_3$, by using Lemma 16.11 in \cite{David:2009}, there exists a 
	minimal cone $Z(z,\rho)$ centered at $z$ of type $\mathbb{Y}$ such that 
	\begin{equation}\label{eq:LY12}
		d_{x,\rho}(E, Z(x,\rho))\leq\tau_2
	\end{equation}
	and for any ball $B(y,t)\subset B(x,\rho)$
	\begin{equation}
		\left\vert\H^2(E\cap B(y,t))-\H^2(Z(x,\rho)\cap B(y,t)) \right\vert \leq
		\tau_2 r_3^2.
	\end{equation}

	For any $\rho \in (\frac{9}{10}r_3, \frac{4r_0}{5}]$, we put
	$r_{\rho}=\left\vert z \right\vert+\rho$, then $r_{\rho}\leq \frac{9r_0}{10}$,
	and  $\left\vert z\right\vert\leq 3\dist(z,L_1)=6r_3$, thus $r_{\rho}<8r\rho$.
	Let $Z^{r_\rho}$ be the sliding minimal cone which is considered in 
	\eqref{eq:LY1}, then we can get that 
	\[
	d_{z,\rho}(F,Z^{r_{\rho}})\leq
	\frac{r_{\rho}}{\rho}d_{0,r_{\rho}}(F,Z^{r_{\rho}})
	\leq \frac{r_{\rho}}{\rho}\tau_1.
	\]
	We take a point $z'\in Z^{r_{\rho}}$ such that $\left\vert z-z'
	\right\vert\leq r_{\rho}\tau_1$, and take
	$Z(z,\rho)=Z^{r_{\rho}}+\Pi_{L_1}(z)-\Pi_{L_1}(z')$, which through point $z$,
	we obtain that 
	\begin{equation}\label{eq:LY13}
		d_{z,\rho}(F,Z(z,\rho))\leq \frac{\left\vert
		\Pi_{L_1}(z)-\Pi_{L_1}(z') \right\vert}{\rho}+d_{z,\rho}(F,Z^{r_{\rho}})
		\leq \frac{2r_{\rho}}{\rho}\tau_1\leq 16\tau_1.
	\end{equation}
	It follows from \eqref{eq:LY12} and \eqref{eq:LY13} that, when 
	$z\in B(0,r_0/10)\cap F\setminus L_1$ and
	$\dist(z,L_1)\geq \left\vert z \right\vert/3$ with $\theta_F(z)=\frac{3}{2}$, 
	for any $\rho\in (0,4r_0/5]$, we can find a cone $Z(z,\rho)$ such that 
	\begin{equation}\label{eq:LY14}
		d_{z,\rho}(F,Z(z,\rho))\leq \tau_2,
	\end{equation}
	where $Z(z,\rho)$ is a minimal cone when $\rho$ small, and $Z(z,\rho)$ is a 
	sliding minimal cone with boundary $L_1$ when $\rho$ large.

	Case 2. If $\theta_F(z)=1$. We set 
	\[
	E_{Y}=\{ 0 \}\cup \left\{ x\in F\midd \theta(x)=\frac{3}{2} \right\},
	\]
	and denote $\ell_{Y}(z)=\dist(z,E_{Y})$. 
	Using the same argument as in Lemma 16.25 in \cite[in page 205]{David:2009},
	we get that for $\rho\in (0,\ell_Y(z)/3]$, there is a plane
	$Z(z,\rho)$ through $x$ such that 
	\begin{equation}\label{eq:LY15}
		d_{z,\rho}(F,Z(z,\rho))\leq \tau.
	\end{equation}
	For $\rho\in (\ell_{Y}(z)/3, r_0/10]$, we take a point $x\in E_Y$ such that 
	\[
	\left\vert z-x \right\vert\leq \frac{11}{10}\ell_Y(z),
	\]
	and consider the cone $Z(x,r_{\rho})$ as in \eqref{eq:LY14}, where
	$r_{\rho}=\left\vert z-x \right\vert+\rho$. We can get that 
	\[
	d_{z,\rho}(F,Z(x,r_{\rho}))\leq
	\frac{r_{\rho}}{\rho}d_{x,r_{\rho}}(F,Z(x,r_{\rho}))\leq
	\frac{r_{\rho}}{\rho}\tau_2.
	\]
	By a similar argument as before, we can find a cone $Z(z,\rho)$ which is a
	translation of $Z(x,r_{\rho})$ such that 
	\begin{equation}\label{eq:LY16}
		d_{z,\rho}(F,Z(z,\rho))\leq \frac{2r_{\rho}}{\rho}\tau_2<10\tau_2.
	\end{equation}
	It follows that, when $z\in B(0,r_0/10)\cap F\setminus L_1$ and
	$\dist(z,L_1)\geq \left\vert z \right\vert/3$ with $\theta_F(z)=1$, for any
	$\rho\in (0,r_0/10]$, we can find a cone $Z(z,\rho)$ such that 
	\begin{equation}\label{eq:LY17}
		d_{z,\rho}(F,Z(z,\rho))\leq \tau,
	\end{equation}
	where $Z(z,\rho)$ is a minimal cone when $\rho$ is small, and $Z(z,\rho)$ is a 
	sliding minimal cone when $\rho$ is large.

	From inequalyties \eqref{eq:LY7}, \eqref{eq:LY18}, \eqref{eq:LY14} and
	\eqref{eq:LY17}, we can say that for any $z\in B(0,r_0/10)$, and for any
	$\rho\in (0,r_0/10]$, there is a cone $Z(z,\rho)$ such that 
	\[
	d_{z,\rho}(F,Z(z,\rho))\leq \tau,
	\]
	where $Z(z,\rho)$ is a minimal cone when $\rho<\dist(z,L_1)$, and
	$Z(z,\rho)$ is a sliding minimal cone when $\rho\geq \dist(z,L_1)$.

	By Corollary \ref{co:PT}, we can get our desired result.
\end{proof}
\begin{lemma}\label{lemmay2}
	Let $\Omega$, $L_1$ be as in \eqref{simple}, $U$ an open set with $0\in U$.
	Suppose that $E\subset \Omega$ is an $(U,h)$-sliding-almost-minimal 
	set. If $\theta_E(0)=\frac{7}{4}$, then for each small $\tau>0$, we can 
	find a radius $r>0$, a biH\"older map $\phi: B(0,3r/2)\cap\Omega\to
	B(0,2r)\cap\Omega$ and a sliding minimal cone $Z$ of type $\mathbb{Y}_{+}$
	such that
	\begin{equation}\label{eq:lemmay2}
		\begin{gathered}
			\phi(x)\in L_1\text{ for }x\in L_1,\| f-\id \|_{\infty}\leq \tau,\\
			(1+\tau)^{-1}\left\vert z-y \right\vert^{1+\tau}\leq \left\vert
			\phi(z)-\phi(y) \right\vert\leq (1+\tau)\left\vert z-y
			\right\vert^{\frac{1}{1+\tau}},\\
			B(0,r)\cap\Omega\subset \phi\left( B\left( 0,\frac{3r}{2} \right)\cap
			\Omega \right)\subset B(0,2r)\cap \Omega,\\
			E\cap B(0,r)\subset \phi\left( (Z\cup L_1)\cap B\left( 0,\frac{3r}{2}
			\right)\right)\subset E\cap B(0,2r).
		\end{gathered}
	\end{equation}
\end{lemma}
\begin{proof}
	We put $F=\overline{E\setminus L_1}$, then $F$ is also
	$(U,\delta,h)$-sliding-almost-minimal. By lemma \ref{lemmay}, for each small
	$\tau>0$, we can find $r>0$, a biH\"older map $\phi: B(0,3r/2)\cap\Omega\to
	B(0,2r)\cap\Omega$ and a sliding minimal cone $Z$ of type $\mathbb{Y}_{+}$ 
	such that
	\begin{equation}\label{eq:lemmay3}
		\begin{gathered}
			\phi(x)\in L_1\text{ for }x\in L_1,\| f-\id \|_{\infty}\leq \tau,\\
			(1+\tau)^{-1}\left\vert x-y \right\vert^{1+\tau}\leq \left\vert
			\phi(x)-\phi(y) \right\vert\leq (1+\tau)\left\vert x-y
			\right\vert^{\frac{1}{1+\tau}},\\
			B(0,r)\cap\Omega\subset \phi\left( B\left( 0,\frac{3r}{2} \right)\cap
			\Omega \right)\subset B(0,2r)\cap \Omega,\\
			F\cap B(0,r)\subset \phi\left(Z\cap B\left( 0,\frac{3r}{2}
			\right)\right)\subset F\cap B(0,2r).
		\end{gathered}
	\end{equation}
	Thus 
	\[
	E\cap B(0,r)\subset\phi\left( (Z\cup L_1)\cap B\left( 0,\frac{3r}{2} \right)
	\right)\subset E\cap B(0,2r).
	\]
\end{proof}
\begin{remark}\label{re:y}
	Suppose that $\Omega$, $L_1$ and $U$ are as in Lemma \ref{lemmay}, and that
	$E\subset \Omega$ is a $(U,h)$-sliding-almost-minimal set with
	$\theta_E(0)=\frac{3}{4}$ or $\frac{7}{4}$. If $\tau\in (0,1)$ is small
	enough, we can find $\varepsilon'(\tau)>0$ such that when $r>0$ is
	such that 
	\[
	B(0,10r)\subset U, h(20r)\leq \varepsilon'(\tau),
	\int_{0}^{20r}\frac{h(t)}{t}dt\leq\varepsilon'(\tau),\theta_E(0,10r)\leq
	\theta_E(0)+\varepsilon'(\tau),
	\]
	then for any $x\in E\cap B(0,r)$ and any $0<t\leq 2r$, we can find a minimal
	cone or sliding minimal cone $Z(x,t)$ such that 
	\[
	d_{x,t}(E,Z(x,t))\leq \tau,
	\]
	where $Z(x,t)$ is a minimal cone when $0<t<\dist(x,L_1)$, and $Z(x,t)$ is a
	sliding minimal cone when $\dist(x,L_1)\leq t\leq 2r$.
\end{remark}
Indeed, we can take $\tau\in (0,1)$ be such that 
	\[
	\left( \frac{1}{2}+\varepsilon(\tau)
	\right)e^{\alpha\varepsilon(\tau)}<\frac{3}{4} \text{ and } \left(
	\frac{3}{2}+\varepsilon(\tau)e^{\lambda\varepsilon(\tau)} \right)<d_T,
	\]
	then we take 
	\[
	\tau_1=\min\left\{\frac{\tau}{10^4},
	\frac{1}{100}\varepsilon\left(\frac{\tau}{100}\right)\right\},
	\]
	and let $\varepsilon(\tau_1)$ be as in Lemma \ref{mainlemma}. We can
	check from the proof of Lemma \ref{lemmay} that
	$\varepsilon'(\tau)=\varepsilon(\tau_1)$ is the number what we desire.
\begin{proposition}\label{prop:PRB}
	Let $\Omega$, $L_1$ be as in \eqref{simple}, $U$ an open set.
	Let $E\subset\Omega$ be an $(U,h)$-sliding-almost-minimal set, and $x\in
	L_1\cap U$ be a point. If $\theta_E(x)\in\{
	1/2,3/4,3/2,7/4 \}$, then for each small $\tau>0$, we can find a radius 
	$r>0$, a sliding minimal cone $Z$ centered at $x$ and a biH\"older map
	$\phi: B(x,3r/2)\cap \Omega\to B(x,2r)\cap\Omega$ such that 
	\begin{equation}\label{eq:PRB}
		\begin{gathered}
			\phi(z)\in L_1\text{ for }z\in L_1,\| f-\id \|_{\infty}\leq \tau,\\
			(1+\tau)^{-1}\left\vert z-y \right\vert^{1+\tau}\leq \left\vert
			\phi(z)-\phi(y) \right\vert\leq (1+\tau)\left\vert z-y
			\right\vert^{\frac{1}{1+\tau}},\\
			B(0,r)\cap\Omega\subset \phi\left( B\left( 0,\frac{3r}{2} \right)\cap
			\Omega \right)\subset B(0,2r)\cap \Omega,\\
			E\cap B(0,r)\subset \phi\left(Z\cap B\left( 0,\frac{3r}{2}
			\right)\right)\subset E\cap B(0,2r).
		\end{gathered}
	\end{equation}
	In addition, if $\theta_E(x)=\frac{1}{2}$, $Z$ is a cone of type
	$\mathbb{P}_{+}$; if $\theta_E(x)=\frac{3}{4}$, $Z$ is a cone of type
	$\mathbb{Y}_{+}$; if $\theta_E(x)=\frac{3}{2}$, $Z=Z'\cup L_1$ where $Z'$ 
	is a cone of type $\mathbb{P}_{+}$; if $\theta_E(x)=\frac{7}{4}$, $Z=Z'\cup
	L_1$ where $Z'$ is a cone of type $\mathbb{Y}_{+}$.
\end{proposition}
The proof immediately follows from Lemma \ref{lemmaplane}, Lemma
\ref{le:PUFP2}, Lemma \ref{lemmay} and Lemma \ref{lemmay2}.
\begin{corollary}\label{co:regularity}
	Let $\Omega$, $L_1$ be as in \eqref{simple}, $U$ an open set.
	Let $E\subset\Omega$ be an $(U,h)$-sliding-almost-minimal set with
	$E\supset L_1$. Then for each small $\tau>0$ and each $x\in L_1\cap U$, we
	can find a radius $r>0$, a sliding minimal cone $Z$ and a biH\"older map
	$\phi: B(x,3r/2)\to B(x,2r)$ such that 
	\begin{equation}\label{eq:CRT}
		\begin{gathered}
			\phi(x)\in L_1\text{ for }x\in L_1, \| f-\id \|_{\infty}\leq \tau,\\
			(1+\tau)^{-1}\left\vert x-y \right\vert^{1+\tau}\leq \left\vert
			\phi(x)-\phi(y) \right\vert\leq (1+\tau)\left\vert x-y
			\right\vert^{\frac{1}{1+\tau}},\\
			B(0,r)\cap\Omega\subset \phi\left( B\left( 0,\frac{3r}{2} \right)\cap
			\Omega \right)\subset B(0,2r)\cap \Omega,\\
			E\cap B(0,r)\subset \phi\left(Z\cap B\left( 0,\frac{3r}{2}
			\right)\right)\subset E\cap B(0,2r).
		\end{gathered}
	\end{equation}
\end{corollary}
\begin{proof}
	Since $E\supset L_1$, any blow-up limit $F$ of $E$ at $x$ contains $L_1$, so it
	is a sliding minimal cone contains $L_1$. By Theorem
	\ref{thm:MCCB}, we can get that $F=L_1$ or 
	$F=Z\cup L_1$, where $Z$ is a cone of type $\mathbb{P}_{+}$ or
	$\mathbb{Y}_{+}$. If $F=L_1$, by Lemma \ref{le:D1}, then there exists
	a ball $B(x,r)$ such that $E\cap B(x,r)=L_1\cap B(x,r)$, thus
	\eqref{eq:CRT} hold automatically. If $F\neq L_1$, then $F=Z\cup L_1$ where
	$Z$ is a sliding minimal cone of type $\mathbb{P}_{+}$ or $\mathbb{Y}_{+}$;
	we get that $\theta_E(x)=\frac{3}{2}$ or $\frac{7}{4}$, by Proposition 
	\ref{prop:PRB}, we obtain the conclusion.
\end{proof}

\section{Regularity of sliding almost minimal sets
\uppercase\expandafter{\romannumeral2}}

In the previous section, we get some regularity for sliding almost
minimal sets with whose boundary is a plane. In this section we will
give a similar result, but with where the boundary is a $C^1$ manifold. 

Let $\Sigma\subset\mathbb{R}^{3}$ be a connected closed set such that the boundary
$\partial\Sigma$ is a $2$-dimensional $C^1$ manifold. 
For any $x\in \partial \Sigma$, the tangent cone of $\Sigma$ at $x$ is a
half space, and the boundary of the half space is the 
tangent plane of $\partial\Sigma$ at $x$.
\begin{theorem}\label{mainthm}
	Let $\Sigma$ be as above, $x\in \partial\Sigma$, $U$ be a neighborhood of
	$x$. Suppose that $E\subset \Sigma$ is an
	$(U,h)$-sliding-almost-minimal set with sliding boundary $\partial\Sigma$
	and that $E\supset \partial\Sigma$. Then for each small $\tau>0$, we can
	find a radius $\rho>0$, a sliding minimal cone $Z$ in $\Omega$ with 
	sliding boundary $L_1$ and a biH\"older map 
	$\phi: B(x,3\rho/2)\cap \Omega\to B(x,2\rho)\cap \Sigma$ such that
	\begin{equation}\label{eq:mainthm2}
		\begin{gathered}
			\phi(x)\in \partial\Sigma\text{ for }x\in L_1,\
			\| \phi-\id \|_{\infty}\leq 3\tau,\\
			C\left\vert x-y \right\vert^{1+\tau}\leq \left\vert
			\phi(x)-\phi(y) \right\vert\leq C^{-1}\left\vert x-y
			\right\vert^{\frac{1}{1+\tau}},\\
			B(x,\rho)\cap \Sigma\subset \phi\left( B\left( x,\frac{3\rho}{2} \right)\cap
			\Omega \right)\subset B(x,2\rho)\cap \Sigma,\\
			E\cap B(x,\rho)\subset \phi\left(Z\cap B\left( x,\frac{3\rho}{2}
			\right)\right)\subset E\cap B(x,2\rho),
		\end{gathered}
	\end{equation}
	where $\Omega$ is the tangent cone of $\Sigma$ at $x$ and $L_1$ is the 
	boundary of $\Omega$.
\end{theorem}

The strategy of the proof will be the same as for Corollary \ref{co:regularity}.
We do not want repeat the whole section above, because
most of the statements and proofs still work. We only give a sketch. 

Firstly, Lemma \ref{mainlemma} is still true when we replace $\Omega$ and
$L_1$ by $\Sigma$ and $\partial\Sigma$ respectively. That is, it can be stated
as follows:
\begin{lemma}\label{mainlemma2}
	Let $\Sigma$ and $\partial\Sigma$ be as in Theorem \ref{mainthm}. Suppose
	that $E\subset\Sigma$ is $(U,h)$-sliding-almost-minimal. If for each 
	$\tau>0$, we can find $\varepsilon_1(\tau)>0$ such that if 
	$x\in E\cap \partial\Sigma$ and $r>0$ are such that 
	\[
	B(x,r)\subset U, h(2r)\leq\varepsilon_1(\tau),
	\int_{0}^{2r}\frac{h(t)dt}{t}\leq\varepsilon_1(\tau),
	\theta_E(x,r)\leq \theta_E(x)+\varepsilon_1(\tau),
	\]
	then for every $\rho\in (0,9r/10]$ there is a sliding minimal cone
	$Z_x^\rho$ such that 
	\[
	d_{x,\rho}(E,Z_x^{\rho})\leq \tau,
	\]
	and for any ball $B(y,t)\subset B(x,\rho)$,
	\[
	|\H^2(E\cap B(y,t))-\H^2(Z_x^{\rho}\cap B(y,t))|\leq \tau\rho^2,
	\]
	where $Z_x^{\rho}$ is a sliding minimal cone in $\Sigma_x$ with sliding
	boundary $T_x$, where we denote by $\Sigma_x$ and $T_x$ the tangent cone of
	$\Sigma$ at $x$ and tangent plane of $\partial\Sigma$ at $x$ respectively.
	If $E\supset \partial\Sigma$, then we can suppose that $Z_x^{\rho}\supset
	T_x$.
\end{lemma}
For each $x\in U\cap \partial\Sigma\cap E$, we see that any blow-up
limit $Z$ of $E$ at $x$ is a sliding minimal cone in $\Sigma_x$ with sliding
boundary $T_x$, see \cite[Theorem 24.13]{David:2014}. 
If $E\supset \partial\Sigma$, we have that $Z\supset T_x$, thus
$Z=T_x$ or $Z=T_x\cup Z'$, where $Z'$ is a sliding minimal cone in $\Sigma_x$ with
sliding boundary $T_x$ of type $\mathbb{P}_{+}$ or $\mathbb{Y}_{+}$. Hence, we
get that $\theta_E(x)=1$, $\frac{3}{2}$ or $\frac{7}{4}$. 

Similar to Lemma \ref{le:D1}, we can get that if $E\supset\partial\Sigma$ is
sliding almost minimal and a blow-up limit of $E$ at $x\in \partial\Sigma$ is
the tangent plane $T_x$ of $\partial\Sigma$ at that point, then there exists
$r>0$ such that $E\cap B(x,r)=\partial\Sigma\cap B(x,r)$. Once we get that, we
can show a similar result to Lemma \ref{le:removeboundary}. That is, if
$E\subset\Sigma$ is $(U,h)$-sliding-almost-minimal and
$E\supset\partial\Sigma$, then, by putting $F=\overline{E\setminus\Sigma}$, we
shall have $\H^2(F\cap \partial\Sigma\cap U)=0$ and $F$ is also
$(U,h)$-sliding-almost-minimal.
Thus, $\theta_F(x)$ can only take two values. That is,
$\frac{1}{2}$ and $\frac{3}{4}$.

Finally, if $E\subset\Sigma$ is sliding almost minimal, 
$x\in E\cap \partial\Sigma$, we can get that if $\theta_E(x)=\frac{1}{2}$ or
$\frac{3}{4}$, then the sliding minimal cone $Z_x^{\rho}$ which is taken in Lemma
\ref{mainlemma2} is of type $\mathbb{P}_{+}$ or $\mathbb{Y}_{+}$; if $E\supset
\partial\Sigma$ and $\theta_E(x)=1$, then $Z_x^{\rho}=T_x$; if
$\theta_E(x)=\frac{3}{2}$ or $\frac{7}{4}$, then $Z_x^{\rho}=T_x\cup Z$, $Z$
is of type $\mathbb{P}_{+}$ or $\mathbb{Y}_{+}$.

We also need a lemma like Lemma \ref{le:BDAP}.
\begin{lemma}
	Let $\Sigma$ and $\partial\Sigma$ be as in Lemma \ref{mainlemma2}, $U$ be an
	open set. Let $E\subset\Sigma$ be a $(U,h)$-sliding-almost-minimal set. Let
	$\varepsilon\in (0,1/100)$ be a small number.
	Suppose that $x\in U\cap \partial\Sigma\cap E$, $\theta_E(x)=\frac{1}{2}$ or
	$\frac{3}{4}$. If $B(x,2r_0)\subset U$, and for any $y\in E\cap
	\partial\Sigma\cap B(x,r_0)$ and any $0<\rho\leq r_0$, there exists a
	sliding minimal cone $Z_y^\rho$ in $\Omega_y$ (the tangent cone of $\Sigma$
	at $y$) with sliding boundary $\partial\Omega_y$ such that 
	\[
	d_{y,r}(E,Z_y^\rho)\leq \varepsilon,
	\]
	then there exists a radius $r>0$ such that for any $z\in E\cap B(x,r)$, we
	can find a point $a\in E\cap B(x,2r)\cap\partial\Sigma$ satisfying 
	\[
	\dist(z,\partial\Sigma)\geq (1-10\varepsilon)|z-a|.
	\]
\end{lemma}

Now, we state a similar result as Lemma \ref{lemmaplane} and Lemma \ref{lemmay},
or rather, a similar result as Remark \ref{re:plane} and Remark \ref{re:y}.
The proof can be adapted from the proof of Lemma \ref{lemmaplane} and
Lemma \ref{lemmay}.
\begin{lemma}\label{le:APPBC}
	Let $\Sigma$ and $\partial\Sigma$ be as in Lemma \ref{mainlemma2}. Let
	$E\subset\Sigma$ be a $(U,h)$-sliding-almost-minimal set such that
	$\theta_E(x)\in\{\frac{1}{2},\frac{3}{2},\frac{3}{4},\frac{7}{4}\}$,
	$x\in E\cap \partial\Sigma\cap U$. If $\tau\in (0,1)$ is a small 
	enough number, then we can find $\varepsilon_2(\tau)>0$ such that when 
	\[
		B(x,10r)\subset U, h(20r)\leq\varepsilon_2(\tau),
		\int_{0}^{20r}\frac{h(t)dt}{t}\leq\varepsilon_2(\tau),
		\theta_E(x,10r)\leq \theta_E(x)+\varepsilon_2(\tau),
	\]
	for some $r>0$, we have that for any $y\in E\cap B(x,r)$, and any $0<t\leq
	2r$, there exists a cone or a sliding minimal cone $Z(y,t)$ satisfying 
	\[
		d_{y,t}(E,Z(y,t))\leq \tau,
	\]
	where $Z(y,t)$ is a cone when $0<t<\dist(x,\partial\Sigma)$, $Z(y,t)$ is a sliding
	minimal cone centered at a point in $B(x,r)\cap \partial\Sigma$ when $\dist(y,\partial\Sigma)\leq t\leq 2r$.
\end{lemma}
\begin{proof}[Proof of Theorem \ref{mainthm}]
Without loss of generality, we assume $x=0$. Let $\Omega$ be the tangent cone
of $\Sigma$ at $0$, $L_1$ be the tangent plane of $\partial\Sigma$ at $0$.
Then $\Omega$ is a half space, and $L_1$ is its boundary. Let $\tau>0$ and
$r>0$ be as in Lemma \ref{le:APPBC}. Since $\partial\Sigma$ is a 2-dimensional
$C^1$ manifold, for any $\varepsilon\in (0,\tau)$,  we can find a radius
$0<R<\frac{r}{2}$ and a $C^1$ diffeomorphism $f:\Omega\cap B(0,R)\to\Sigma$ 
such that $f(0)=0$, $Df(0)=\id$, $f(L_1\cap B(0,R))\subset\partial\Sigma$ and
\[
\|Df(x)-\id\|\leq \varepsilon.
\]
We put 
\[
F=f^{-1}(\Sigma\cap B(0,R)).
\]

For any $x\in F$ and $0<t\leq 2r$, by Lemma \ref{le:APPBC}, we can find a
minimal cone or a sliding minimal cone $Z(f(x),t)$ such that 
\[
d_{f(x),t}(E,Z(f(x),t))\leq \tau,
\]
then
\[
d_{x,(1-\varepsilon)t}\left( f^{-1}(E\cap B(0,R)),Z'(x,t) \right)\leq 
(1+\varepsilon)\tau,
\]
where we assume that $Z(f(x),t)$ is centered at $a$, and denote 
\[
Z'(x,t)=Df^{-1}(x)\left( Z(f(x),t)-a \right)+f^{-1}(a).
\]
We note from Lemma \ref{le:APPBC} that if 
$Z(f(x),t)$ is a sliding minimal cone, then it is centered at a point in
$\partial\Sigma$. Thus $a\in B(0,R)\cap\partial\Sigma$, and $Z(f(x),t)$ is a
sliding minimal cone in $\Sigma_a$ with sliding boundary $T_a$.

Since $\|Df(x)-\id\|\leq \varepsilon$, we have $\|Df^{-1}(x)-\id\|\leq
2\varepsilon$. We take $Z''(x,t)=Z(f(x),t)-a+f^{-1}(a)$, then
\[
d_{x,(1-\varepsilon)t}(Z'(x,t)),Z''(x,t))\leq 2\varepsilon,
\]
thus
\[
d_{x,(1-\varepsilon)t}(F,Z''(x,t))\leq (1+\varepsilon)\tau+2\varepsilon.
\]
$Z''(x,t)$ is a minimal cone or a sliding minimal. 

Let $\mathcal{T}_a:\mathbb{R}^{3}\to\mathbb{R}^{3}$ be the translation 
which send point $z$ to $z-a+f^{-1}(a)$. Then
$Z''(x,t)=\mathcal{T}_a(Z(f(x),t))$. If $Z(f(x),t)$ is a sliding minimal cone,
then $Z''(x,t)$ is a slding minimal cone in $\mathcal{T}_a(\Sigma_a)$ with
sliding boundary $\mathcal{T}_a(T_a)$.
We put $y=f^{-1}(a)$, then it is quite easy to see that $Df(y)$ maps
$\Omega$ and $L_1$ to $\mathcal{T}_a(\Sigma_a)$ and $\mathcal{T}_a(T_a)$
respectively. Since $\|Df(y)-\id\|\leq \varepsilon$, we can find a rotation
$\mathcal{R}_y$ centered at point $y$, which will rotate 
$\mathcal{T}_a(\Sigma_a)$ and $\mathcal{T}_a(T_a)$ to $\Omega$ and $L_1$
respectively, such that 
\[
d_{0,1}(\mathcal{R}_y(Z''(x,t)),Z''(x,t))\leq 2\varepsilon,
\]
then $\mathcal{R}(Z''(x,t))$ is a sliding minimal cone in $\Omega$ with 
sliding boundary $L_1$. 

We take $Z(x,t)=\mathcal{R}(Z''(x,t))$ when $Z''(x,t)$ is a slding minimal 
cone, and take $Z(x,t)=Z''(x,t)$ when $Z''(x,t)$ is a minimal cone, then 
\[
d_{x,(1-\varepsilon)t}(F,Z(x,t))\leq (1+\varepsilon)\tau+5\varepsilon<7\tau.
\]

By Corollary \ref{co:PT}, we can find a radius $r'\in (0,R/2)$, a sliding minimal cone
$Z$ in $\Omega$ with sliding boundary $L_1$, and a biH\"older map $\varphi:
B(0,3r'/2)\cap\Omega\to B(0,2r')\cap\Omega$ such that 
\[
		\begin{gathered}
			\varphi(x)\in L_1\text{ for }x\in L_1\text{ and } \| \varphi-\id \|_{\infty}\leq
			\tau,\text{ and}\\
			(1+\tau)^{-1}\left\vert x-y \right\vert^{1+\tau}\leq \left\vert
			\varphi(x)-\varphi(y) \right\vert\leq (1+\tau)\left\vert x-y
			\right\vert^{\frac{1}{1+\tau}},\\
			B(0,r')\cap\Omega\subset \varphi\left( B\left( 0,\frac{3r'}{2} \right)\cap
			\Omega \right)\subset B(0,2r')\cap \Omega,\\
			F\cap B(0,r')\subset \varphi\left(Z\cap B\left( 0,\frac{3r'}{2}
			\right)\right)\subset F\cap B(0,2r').
		\end{gathered}
\]
We now take $\phi=f\circ\varphi$. Then $\phi:B(0,3r'/2)\cap \Omega\to \Sigma$
is a biH\"older map, and we can easily check that \eqref{eq:mainthm2} hold if
we take $\rho=r'/2$.
\end{proof}
\section{Existence of two dimensional singular minimizers}
Let $\Sigma\subset\mathbb{R}^{3}$ be a connected closed set such that the boundary
$\partial\Sigma$ is a $2$-dimensional connected compact $C^1$ manifold.
Let $G$ be any abelian group, $L$ be a subgroup of the \v Cech homology group 
$\check{H}_1(\partial\Sigma;G)$. We say a compact set $E\supset \partial\Sigma$
spans $L$ if $L$ is contained in the kernel of the homomorphism induced by the
inclusion map $\partial\Sigma\to E$. We set
\[
\mathscr{C}=\{ E\subset\Sigma\mid E\text{ spans }L\}.
\]
From paper \cite{Fang:2013}, we see that there exist a set $E_0\in
\mathscr{C}$, we call it a \v Cech minimizer, such that 
\begin{equation}\label{eq:minimizer}
	\H^2(E_0\setminus\partial\Sigma)=
	\inf_{E\in\mathscr{C}}\H^2(E\setminus\partial\Sigma).
\end{equation}
Let's check that $E_0$ is also sliding minimal with boundary
$\partial\Sigma$. Let $\{ \varphi_t \}_{0\leq t\leq 1}$ be any
sliding-deformation in $\Sigma$. We put $F=\varphi_1(E_0)$, denote by
$i:\partial\Sigma\to E_0$ and $j:\partial\Sigma\to F$ the inclusion maps. 
We consider the map 
\[
\psi:\partial\Sigma\times [0,1]\to F, \psi(x,t)=\varphi_t(x).
\]
It is continuous, and $\psi(x,0)=j(x)$,
$\psi(x,1)=\varphi\vert_{\partial\Sigma}(x)$, thus the maps
$j:\partial\Sigma\to F$ and $\varphi\vert_{\partial\Sigma}:\partial\Sigma\to
F$ are homotopy equivalent. Then $j_{\ast}=(\varphi\vert_{\partial 
\Sigma})_{\ast}$, where for any map between two topology
spaces $f:X\to Y$, we denote by $f_{\ast}$ the homomorphism 
$\check{H}_1(X;G)\to \check{H}_1(Y;G)$ induced by the map $f$.
However, we know that $\varphi_1\vert_{B}=\varphi_1\vert_{E_0}\circ i$, thus 
\[
j_{\ast}=(\varphi_1\vert_{B})_{\ast}=(\varphi_1\vert_{E_0})_{\ast}\circ
i_{\ast}.
\]
But we know that $i_{\ast}(L)=0$, thus $j_{\ast}(L)=0$, and $F\in
\mathscr{C}$, so 
\[
\H^2(F\setminus\partial\Sigma)\geq \H^2(E_0\setminus\partial \Sigma),
\]
$E_0$ is sliding minimal.

We now consider an analogous topic, that replace \v Cech homology by
singular homology. Since $\partial\Sigma$ is a two dimensional $C^1$
manifold, the singular homology groups and \v Cech homology groups coincide,
that is, $H_1(\partial\Sigma;G)=\check{H}_1(\partial\Sigma;G)$. We say that
a compact subsets $E\supset\partial\Sigma$ spans $L$ in singular homology,
if $L$ is contained in the kernel of the homomorphism 
$H_1(\partial\Sigma;G)\to H_1(E;G)$ induced by the inclusion map
$\partial\Sigma\to E$. We consider another collection of compact sets 
\[
\mathscr{S}=\{ E\mid E\text{ spans }L \text{ in singular homology} \}.
\]
It is quite easy to see that $\mathscr{S}\subset\mathscr{C}$, that is
because there is a canonical homomorphism from singular homology group to 
\v Cech homology group $H_1(E;G)\to \check{H}_1(E;G)$, and the following 
diagram commutes:
\[
\xymatrix{
H_1(\partial\Sigma;G)\ar[r]\ar@{=}[d]&H_1(E;G)\ar[d]\\
\check{H}_1(\partial\Sigma;G)\ar[r]&\check{H}_1(E;G).
}
\]
If $E$ spans $L$ in singular homology, then from the above commutative diagram,
we can get that $E$ spans $L$ in \v Cech homology, thus $\mathscr{S}\subset
\mathscr{C}$. Our goal is to find a singular minimizer, that is, we want to
find a set $E\in \mathscr{S}$, we call it a singular minimizer, such that 
\[
\H^2(E\setminus\partial\Sigma)=\inf_{F\in
\mathscr{S}}\H^2(F\setminus\partial\Sigma).
\]
\begin{proposition}
	Let $\Sigma$, $G$, $L$ be as above. Then there exists a singular minimizer.
\end{proposition}
\begin{proof}
	Let $E_0$ be a \v{C}ech minimizer. 
	We know, from above discussion, that $E_0$ is sliding minimal. Thus for
	any $x\in E_0$, there is a neighborhood of $x$ where $E_0$ is biH\"older 
	equivalent to minimal cone and by a biH\"older mapping that preserves
	$\partial\Sigma$. By a same
	argument as in \cite[Section 6]{David:2013}, we conclude that $E_0$ is 
	H\"older neighborhood retract. Let's check that $E_0$ is a singular
	minimizer, It is sufficient to show that $E_0$ spans $L$ in
	singular homology. Indeed, the canonical homomorphism $H_1(E_0;G)\to
	\check{H}_1(E_0;G)$ is an isomorphism since $E_0$ is neighborhood retract,
	see for example \cite{Sibe:1959,ES:1952}. Now $E_0$ is a \v{C}ech minimizer, 
	$E_0$ spans $L$ in \v{C}ech homology, thus $E_0$ spans $L$ in singular 
	homology, and we get the conclusion.
\end{proof}


\Addresses
\end{document}